\documentclass{article}

\usepackage[utf8]{inputenc} 
\usepackage[T1]{fontenc}    
\usepackage{hyperref}       
\usepackage{url}            
\usepackage{booktabs}       
\usepackage{amsfonts}       
\usepackage{nicefrac}       
\usepackage{microtype}      

\usepackage{geometry}
\title{An Improved Analysis of Stochastic Gradient Descent with Momentum}

\author{Yanli Liu\footnote{Department of Mathematics, University of California. \newline E-mails: yanli / wotaoyin@math.ucla.edu} \and Yuan Gao\footnote{Department of Industrial Engineering and Operations Research, Columbia University \newline E-mail: gao.yuan@columbia.edu} \and Wotao Yin$^*$
}
\usepackage{amsmath,amssymb,bbm,amsthm}
\usepackage{color}
\usepackage{mathtools}

\usepackage{microtype}
\usepackage{graphicx}
\usepackage{subfigure}
\usepackage{booktabs} %
\usepackage{multicol, multirow}
\usepackage{float}

\usepackage{caption}
\usepackage{hyperref}

\usepackage{dsfont}

\usepackage{bbold}

\usepackage{dsfont}

\usepackage{bbold}

\DeclareMathOperator*{\R}{\mathbb{R}}
\DeclareMathOperator*{\E}{\mathbb{E}}

\DeclareMathOperator*{\Rd}{\mathbb{R}^d}

\newcommand{\tg}{\tilde{g}}

\newcommand{\EE}{\mathbb{E}}

\newtheorem{lemma}{Lemma}
\newtheorem{theorem}{Theorem}
\newtheorem{remark}{Remark}
\newtheorem{corollary}{Corollary}
\newtheorem{definition}{Definition}
\newtheorem{assumption}{Assumption}
\newtheorem{proposition}{Proposition}

\makeatletter
\newcounter{procedure}

\makeatother

\usepackage{microtype}
\usepackage{graphicx}
\usepackage{subfigure}
\usepackage{booktabs} %

\usepackage{hyperref}

\usepackage{amsmath,amssymb,bbm,amsthm}
\usepackage{color}
\usepackage{mathtools}

\usepackage{dsfont}

\usepackage{bbold}

\usepackage{algorithm, algorithmic}

\begin{document}

\maketitle









%

%
%
%
%
%

%

%
\begin{abstract}
SGD with momentum (SGDM) has been widely applied in many machine learning tasks, and it is often applied with dynamic stepsizes and momentum weights tuned in a stagewise manner. Despite of its empirical advantage over SGD, the role of momentum is still unclear in general since previous analyses on SGDM either provide worse convergence bounds than those of SGD, or assume Lipschitz or quadratic objectives, which fail to hold in practice. Furthermore, the role of dynamic parameters have not been addressed. In this work, we show that SGDM converges as fast as SGD for smooth objectives under both strongly convex and nonconvex settings. We also establish \textit{the first} convergence guarantee for the multistage setting, and show that the multistage strategy is beneficial for SGDM compared to using fixed parameters. Finally, we verify these theoretical claims by numerical experiments.

\end{abstract}

\section{Introduction}






%
%
Stochastic gradient methods have been a widespread practice in machine learning. They aim to minimize the following empirical risk: 
\begin{align}
\label{equ: objective}
\min_{x\in\Rd} f(x)\coloneqq \frac{1}{n}\sum_{i=1}^n \ell(x,q_i),
\end{align}
where $\ell$ is a loss function and $\{q_i\}_{i=1}^n$ denotes the training data, $x$ denotes the trainable parameters of the machine learning model, e.g., the weight matrices in a neural network.

In general, stochastic gradient methods can be written as
\begin{align}
\begin{split}
\label{equ: momentum}
m^k &=\beta m^{k-1} + (1-\beta)\tilde{g}^k,\\
x^{k+1} &= x^k - \alpha m^k.
\end{split}
\end{align}
where $\alpha>0$ is a stepsize, $\beta\in[0,1)$ is called momentum weight, and $m^0 = 0$. The classical Stochastic Gradient Descent(SGD) method \cite{robbins1951stochastic} uses $\beta=0$ and $m^k=\tg^k$, where $\tg^k$ is a stochastic gradient of $f(x)$ at $x^k$.  To boost the practical performance, one often applies a momentum weight of $\beta>0$. and the resulting algorithm is often called SGD with momentum (SGDM). SGDM is very popular for training neural networks with remarkable empirical successes, and has been implemented as the default SGD optimizer in Pytorch \cite{paszke2019pytorch} and Tensorflow \cite{abadi2016tensorflow}\footnotemark[1]. 
\footnotetext[1]{Their implementation of SGDM does not have the $(1-\beta)$ before $\tg^k$, which gives $m^k = \sum_{i=1}^k \beta^{k-i}\tg^i$, while $m^k = (1-\beta)\sum_{i=1}^k \beta^{k-i}\tg^i$ for \eqref{equ: momentum}. Therefore, they only differ by a constant scaling.}


%

%
%

%
%

The idea behind SGDM originates from Polyak's heavy-ball method \cite{polyak1964some} for deterministic optimization.
For strongly convex and smooth objectives, heavy-ball method enjoys an accelerated linear convergence rate over gradient descent \cite{ghadimi2015global}. However, the theoretical understanding of its stochastic counterpart is far from being complete. 

In the case of fixed stepsize and momentum weight, most of the current results only apply to restrictive settings. In \cite{loizou2017linearly, loizou2017momentum} and \cite{kidambi2018insufficiency}, the behavior of SGDM on least square regression is analyzed and linear convergence is established. \cite{gitman2019understanding} analyzes the local convergence rate of SGDM for strongly convex and smooth functions, where the initial point $x^0$ is assumed to be close enough to the minimizer $x^*$. \cite{yan2018unified} provides global convergence of SGDM, but only for objectives with \emph{uniformly bounded gradients}, thus excluding many machine learning models such as Ridge regression. Very recently, \cite{yu2019linear} presents a convergence bound of $\mathcal{O}(\frac{1}{k\alpha}+\frac{\alpha}{1-\beta})$ for general smooth nonconvex objectives\footnotemark[3]. When $\beta=0$, this recovers the classical convergence bound of $\mathcal{O}(\frac{1}{k\alpha}+\alpha)$ of SGD \cite{bottou2018optimization}. However, the size of stationary distribution $\mathcal{O}(\frac{\alpha}{1-\beta})$ is $\frac{1}{1-\beta}$ times larger than that of SGD. This factor is not negligible, especially when large $\beta$ values such as $0.99$ and $0.995$ is applied \cite{sutskever2013importance}. Therefore, their result does not explain the competitiveness of SGDM compared to SGD. Concurrent to this work, \cite{sebbouh2020convergence} shows that SGDM
converges as fast as SGD under convexity and strong convexity, and that it is asymptotically faster than SGD for overparameterized models. Remarkably, their analysis considers a different stepsize and momentum weight schedule from this work, and applies to arbitrary sampling without assuming the bounded variance of the gradient noise. 

\footnotetext[3]{Here $k$ is the number of iterations. Note that in \cite{yu2019linear}, a different but equivalent formulation of SGDM is analyzed; their stepsize $\gamma$ is effectively $\frac{\alpha}{1-\beta}$ in our setting.}


In deep learning, SGDM is often applied with various parameter tuning rules to achieve efficient training. One of the most widely adopted rules is called “constant and drop", where a constant stepsize is applied for a long period and is dropped by some constant factor to allow for refined training, while the momentum weight is either kept unchanged (usually $0.9$) or gradually increasing. We call this strategy Multistage SGDM and summarize it in Algorithm \ref{alg: Multistage SGDM}. Practically, (multistage) SGDM was successfully applied to training large-scale neural networks \cite{krizhevsky2012imagenet, hinton2012deep}, and it was found that appropriate parameter tuning leads to superior performance \cite{sutskever2013importance}. Since then, (multistage) SGDM has become increasingly popular \cite{sun2019optimization}. 


At each stage, Multistage SGDM (Algorithm \ref{alg: Multistage SGDM}) requires three parameters: stepsize, momentum weight, and stage length. In \cite{ghadimi2013optimal} and \cite{hazan2014beyond}, doubling argument based rules are analyzed for SGD on strongly convex objectives, where the stage length is doubled whenever the stepsize is halved. Recently, certain stepsize schedules are shown to yield faster convergence for SGD on nonconvex objectives satisfying growth conditions \cite{yuan2019stagewise, davis2019stochastic}, and a nearly optimal stepsize schedule is provided for SGD on least square regression \cite{ge2019step}. These results consider only the momentum-free case. Another recent work focuses on the asymptotic convergence of SGDM (i.e., without convergence rate) \cite{gitman2019understanding}, which requires the momentum weights to approach either $0$ or $1$, and therefore contradicts the common practice in neural network training. In summary, the convergence rate of Multistage SGDM (Algorithm \ref{alg: Multistage SGDM}) has not been established except for the momentum-free case, and the role of parameters in different stages is unclear.

\begin{algorithm}[H]
\caption{Multistage SGDM}
\label{alg: Multistage SGDM}
    \textbf{Input:} problem data $f(x)$ as in \eqref{equ: objective}, number of stages $n$, momentum weights $\{\beta_i\}_{i=1}^n\subseteq [0,1)$, step sizes $\{\alpha_i\}_{i=1}^n$, and stage lengths $\{T_i\}_{i=1}^n$ at $n$ stages, initialization $x^1\in\Rd$ and $m^0=0$, iteration counter $k=1$.
    \begin{algorithmic}[1]
       \FOR{$i=1,2,...,n$}
        \STATE{$\alpha\leftarrow \alpha_i, \beta\leftarrow \beta_i$;}
           \FOR{$j=1,2,...,T_i$}
           \STATE{Sample a minibatch $\zeta^k$ uniformly from the training data;}
           \STATE{$\tilde{g}^{k}\leftarrow \nabla_x l(x^{k}, \zeta^k)$;}
           \STATE{$m^{k}\leftarrow \beta m^{k-1}+(1-\beta)\tilde{g}^{k}$;}
           \STATE{$x^{k+1}\leftarrow x^k-\alpha m^k$;}
           \STATE{$k\leftarrow k+1$;}
           \ENDFOR
       \ENDFOR
       \RETURN{$\tilde{x}$, which is generated by first choosing a stage $l\in\{1,2,...n\}$ uniformly at random, and then choosing $\tilde{x}\in\{x^{T_1+...+T_{l-1}+1}, x^{T_1+...+T_{l-1}+2}, ...,x^{T_1+...+T_{l}}\}$ uniformly at random;}
    \end{algorithmic}
\end{algorithm}

\subsection{Our contributions}
In this work, we provide new convergence analysis for SGDM and Multistage SGDM that resolve the aforementioned issues. A comparison of our results with prior work can be found in Table \ref{table: comparison}.
\begin{enumerate}
    \item We show that for both strongly convex and nonconvex objectives, SGDM \eqref{equ: momentum} enjoys the same convergence bound as SGD.
    This helps explain the empirical observations that SGDM is at least as fast as SGD \cite{sun2019optimization}.  Our analysis relies on a new observation that, the update direction $m^k$ of SGDM \eqref{equ: momentum} has a controllable deviation from the current full gradient $\nabla f(x^k)$, and enjoys a smaller variance. Inspired by this, we construct a new Lyapunov function that properly handles this deviation and exploits an auxiliary sequence to take advantage of the reduced variance. 
    
    Compared to aforementioned previous work, our analysis applies to not only least squares, does not assume uniformly bounded gradient, and improves the convergence bound.
    
    \item For the more popular SGDM in the multistage setting (Algorithm \ref{alg: Multistage SGDM}), we establish its convergence and demonstrate that the multistage strategy are faster at initial stages. Specifically, we allow larger stepsizes in the first few stages to boost initial performance, and smaller stepsizes in the final stages decrease the size of stationary distribution. Theoretically, we properly redefine the aforementioned auxiliary sequence and Lyapunov function to incorporate the stagewise parameters. 
    
    To the best of our knowledge, this is the first convergence guarantee for SGDM in the multistage setting.

    
    
\end{enumerate}

\begin{table*}[ht]
\centering
\begin{tabular}{c|c|c}
\hline 
Method & \begin{tabular}[c]{@{}c@{}}Additional \\ Assumptions\end{tabular} & \begin{tabular}[c]{@{}c@{}}Convergence \\ Bound\end{tabular} \\ \hline\hline
SGDM \cite{yan2018unified}         & Bounded gradient      & $\EE[\|\nabla f(x_{\text{out}})\|^2]=\mathcal{O}\left(\frac{1}{k\alpha}+\frac{\alpha \sigma^2}{1-\beta}\right)$           \\[0.2em] \hline 
SGDM \cite{yu2019linear}             &    -       &     $\EE[\|\nabla f(x_{\text{out}})\|^2]=\mathcal{O}\left(\frac{1}{k\alpha}+\frac{\alpha \sigma^2}{1-\beta}\right)$              \\[0.2em] \hline 
SGDM (*)            &     -      & $\EE[\|\nabla f(x_{\text{out}})\|^2]=\mathcal{O}\left(\frac{1}{k\alpha}+\alpha \sigma^2\right)$                  \\[0.2em] \hline 
SGDM (*)            &  Strong convexity   & $\EE[f(x^k) - f^*]=\mathcal{O}\left((1-\alpha \mu)^k + \alpha \sigma^2\right)$              \\[0.2em] \hline
Multistage SGDM(*) &    -       &   $\EE[\|\nabla f(x_{\text{out}})\|^2] = \mathcal{O}\left(\frac{1}{n A_2}+\frac{1}{n}\sum_{l=1}^n \alpha_l\sigma^2\right)$               \\[0.2em] \hline\hline
\end{tabular}
\caption{Comparison of our results (*) with prior work under Assumption \ref{assump: standard assumption} and additional assumptions. "Bounded gradient" stands for the bounded gradient assumption $\|\nabla f(x)\|\leq G$ for some $G>0$ and all $x\in\Rd$. This work removes this assumption and improves convergence bounds. Strongly convex setting and multistage setting are also analyzed. We omit the results of \cite{ghadimi2013optimal} and \cite{hazan2014beyond} as their analysis only applies to SGD (momentum-free case). }
\label{table: comparison}
\end{table*}

\subsection{Other related work}
Nesterov's momentum achieves optimal convergence rate in deterministic optimization \cite{nesterov2013introductory}, and has also been combined with SGD for neural network training \cite{sutskever2013importance}. Recently, its multistage version has been analyzed for convex or strongly convex objectives \cite{aybat2019universally,kulunchakov2019generic}. Other forms of momentum for stochastic optimization include PID Control-based methods \cite{an2018pid}, Accelerated SGD \cite{kidambi2018insufficiency}, and Quasi-Hyperbolic Momentum \cite{ma2019qh}. In this work, we restrict ourselves to heavy-ball momentum, which is arguably the most popular form of momentum in current deep learning practice.

\section{Notation and Preliminaries}

Throughout this paper, we use $\|\cdot\|$ for  vector $\ell_2$-norm, $\langle\cdot, \cdot\rangle$ stands for dot product. Let $g^k$ denote the full gradient of $f$ at $x^k$, i.e., $g^k\coloneqq \nabla f(x^k)$, and $f^{*}\coloneqq \min_{x\in\Rd}f(x)$.  %

\begin{definition}
\label{def: smoothness}
We say that $f: \Rd \rightarrow \R$ is $L-$smooth with $L\geq 0$, if it is differentiable and satisfies
\[
f(y)\leq f(x)+\langle \nabla f(x), y-x\rangle+\frac{L}{2}\|y-x\|^2, \forall x, y \in\Rd.
\]
We say that $f: \Rd \rightarrow \R$ is $\mu-$strongly convex with $\mu\geq 0$, if it satisfies
\[
f(y)\geq f(x)+\langle \nabla f(x), y-x\rangle+\frac{\mu}{2}\|y-x\|^2, \forall x, y \in\Rd.
\]
\end{definition}
%
%
%
%
%
%
%
%
The following assumption is effective throughout, which is standard in stochastic optimization.
\begin{assumption}
\label{assump: standard assumption}
\begin{enumerate}
    \item \textbf{Smoothness:} The objective $f(x)$ in \eqref{equ: objective} is $L-$smooth.
    \item \textbf{Unbiasedness:} At each iteration $k$, $\tg^k$ satisfies $\E_{\zeta^k} [\tg^k] = g^k$.
    \item \textbf{Independent samples:} the random samples $\{\zeta_k\}_{k=1}^{\infty}$ are independent.
    \item \textbf{Bounded variance:} the variance of $\tg^k$ with respect to $\zeta^k$ satisfies $\mathrm{Var}_{\zeta^k}(\tg^k) = \E_{\zeta^k}[\|\tg^k - g^k\|^2]\leq \sigma^2$ for some $\sigma^2>0$.
\end{enumerate}
\end{assumption}
Unless otherwise noted, all the proof in the paper are deferred to the appendix. 


\section{Key Ingredients of Convergence Theory}
\label{sec: key ingredients}

%
%

%

%

%
In this section, we present some key insights for the analysis of stochastic momentum methods. For simplicity, we first focus on the case of fixed stepsize and momentum weight, and make proper generalizations for Multistage SGDM in App. \ref{App: generalizations for multistage}.

\subsection{A key observation on momentum}
%

%

%
%


In this section, we make the following observation on the role of momentum: 

\begin{center}
\textit{With a momentum weight $\beta\in[0,1)$, the update vector $m^k$ enjoys a reduced ``variance" of $(1-\beta)\sigma^2$, while having a controllable deviation from the full gradient $g^k$ in expectation.}
\end{center}

First, without loss of generality, we can take $m^{0}=0$, and express $m^k$
as %
\begin{align}
m^k &= 
(1-\beta)\sum_{i=1}^k\beta^{k-i}\tilde{g}^i.\label{equ: m^k}
\end{align}
$m^k$ is a moving average of the past stochastic gradients, with smaller weights for older ones\footnotemark[1]. 

we have the following result regarding the ``variance" of $m^k$, which is measured between $m^k$ and its deterministic version $(1-\beta)\sum_{i=1}^k \beta^{k-i}g^i$.
\begin{lemma}
\label{lem: m^k}
Under Assumption \ref{assump: standard assumption}, the update vector $m^k$ in SGDM \eqref{equ: momentum} satisfies
\begin{align*}
\E\left[\left\|m^k - (1-\beta)\sum_{i=1}^k \beta^{k-i}g^i\right\|^2\right]
&\leq \frac{1-\beta}{1+\beta}(1-\beta^{2k})\sigma^2.
\end{align*}
\end{lemma}
Lemma \ref{lem: m^k} follows directly from the property of the moving average. 

On the other hand, $(1-\beta)\sum_{i=1}^k \beta^{k-i}g^i$ is a moving average of all past gradients, which is in contrast to SGD. It seems unclear how far is $(1-\beta)\sum_{i=1}^k \beta^{k-i}g^i$ from the ideal descent direction $g^k$, which could be unbounded unless stronger assumptions are imposed.  Previous analysis such as \cite{yan2018unified} and \cite{gitman2019understanding} make the blanket assumption of bounded $\nabla f$ to circumvent this difficulty. 

\footnotetext[1]{Note the sum of weights $(1-\beta)\sum_{i=1}^k\beta^{k-i}=1-\beta^k\rightarrow 1$ as $k\rightarrow \infty$.}

In this work, we provide a different perspective to resolve this issue.


\begin{lemma}
\label{lem: difference}
Under Assumption \ref{assump: standard assumption}, we have
\begin{align*}
\EE\left[\left\|\frac{1}{1-\beta^k}(1-\beta)\sum_{i=1}^k \beta^{k-i}g^i-g^k\right\|^2\right]\leq \sum_{i=1}^{k-1}a_{k,i}\E[\|x^{i+1}-x^i\|^2],
\end{align*}
where 
\begin{align}
\label{equ: a_k,i}
a_{k, i}=\frac{L^2\beta^{k-i}}{1-\beta^k}\left(k-i+\frac{\beta}{1-\beta}\right).
\end{align}
\end{lemma}
%
%
%
%
%
%
From Lemma \ref{lem: difference}, we know the deviation of $\frac{1}{1-\beta^k}(1-\beta)\sum_{i=1}^k \beta^{k-i}g^i$ from $g^k$ is controllable sum of past successive iterate differences, in the sense that the coefficients $a_{k,i}$ decays linearly for older ones. This inspires the construction of a new Lyapunov function to handle the deviation brought by the momentum, as we shall see next.


%
%

%
\subsection{A new Lyapunov function}


Let us construct the following Lyapunov function for SGDM:
\begin{align}
\label{equ: lyapunov}
L^k = \left(f(z^k)-f^{\star}\right)+\sum_{i=1}^{k-1} c_i \|x^{k+1-i}-x^{k-i}\|^2.
\end{align}
In the Lyapunov function \eqref{equ: lyapunov}, $\{c_i\}_{i=1}^{\infty}$ are positive constants to be specified later corresponding to the deviation described in Lemma \ref{lem: difference}.  Since the coefficients in \eqref{equ: a_k,i} converges linearly to $0$ as $k\rightarrow \infty$, we can choose $\{c_i\}_{i=1}^{\infty}$ in a diminishing fashion, such that this deviation can be controlled, and $L^k$ defined in \eqref{equ: lyapunov} is indeed a Lyapunov function under strongly convex and nonconvex settings (see Propositions \ref{prop: L^k} and \ref{prop: L^k scvx}). 

In \eqref{equ: lyapunov}, $z^k$ is an auxiliary sequence defined as
\begin{align}
\label{equ: z^k}
z^k=
\begin{cases}
x^k\quad\quad\quad\quad\quad\quad\quad\quad k=1,\\
\frac{1}{1-\beta}x^k-\frac{\beta}{1-\beta}x^{k-1}\quad \,k\geq 2.
\end{cases}
\end{align}
This auxiliary sequence first appeared in the analysis of deterministic heavy ball methods in \cite{ghadimi2015global}, and later applied in the analysis of SGDM \cite{yu2019linear, yan2018unified}. It enjoys the following property.
\begin{lemma}
\label{lem: z^k update}
$z^k$ defined in \eqref{equ: z^k} satisfies
\begin{align*}
z^{k+1}-z^k&=-\alpha\tg^k.
\end{align*}
\end{lemma}

Lemma \ref{lem: z^k update} %
indicates that it is more convenient to analyze $z^k$ than $x^k$ since $z^k$ behaves more like a SGD iterate, although the stochastic gradient $\tg^k$ is not taken at $z^k$.




%
%
%
%
%
%
%
%
%

%
%
%
%
%
%
%
%
%
%

%

%
Since the coefficients of the deviation in Lemma \ref{lem: difference} converges linearly to $0$ as $k\rightarrow \infty$, we can choose $\{c_i\}_{i=1}^{\infty}$ in a diminishing fashion, such that this deviation can be controlled. Remarkably, we shall see in Sec. \ref{sec: SGDM theory} that with $c_1 = \mathcal{O}\left(\frac{L}{1-\beta}\right)$, $L^k$ defined in \eqref{equ: lyapunov} is indeed a Lyapunov function under strongly convex and nonconvex settings, and that SGDM converges as fast as SGD.

Now, let us turn to the Multistage SGDM (Algorithm \ref{alg: Multistage SGDM}), which has been very successful in neural network training. However, its convergence still remains unclear except for the momentum-free case. To establish convergence, we require the parameters of Multistage SGDM to satisfy
\begin{align}
\label{equ: principle}
\begin{split}
\frac{\alpha_i\beta_i}{1-\beta_i} &\equiv A_1,\,\,\text{for}\,\, i = 1,2,...n.\\
\alpha_i T_i &\equiv A_2,\,\,\text{for}\,\, i = 1,2,...n.\\
0\leq \beta_1&\leq \beta_2\leq ...\leq \beta_n<1.
\end{split}
\end{align}
where $\alpha_i, \beta_i,$ and $T_i$ are the stepsize, momentum weight, and stage length of $i$th stage, respectively, and $A_1, A_2$ are properly chosen constants. In principle, one applies larger stepsizes $\alpha_i$ at the initial stages, which will accelerate initial convergence, and smaller stepsizes for the final stages, which will shrink the size of final stationary distribution. As a result, \eqref{equ: principle} stipulates that less iterations are required for stages with large stepsizes and more iterations for stages with small stepsizes. Finally, \eqref{equ: principle} requires the momentum weights to be monotonically increasing, which is consistent with what's done in practice \cite{sutskever2013importance}. often, using constant momentum weight also works.


Under the parameter choices in \eqref{equ: principle}, let us define the auxiliary sequence $z^k$ by
\begin{align}
\label{equ: z^k new}
z^k = x^k - A_1 m^{k-1}.
\end{align}
This $\{z^k\}_{k=1}^{\infty}$ sequence reduces to \eqref{equ: z^k} when a constant stepsize and momentum weight are applied. Furthermore, the observations made in Lemmas \ref{lem: m^k}, \ref{lem: difference}, and \ref{lem: z^k update} can also be generalized (see Lemmas \ref{lem: m^k for multistage}, \ref{lem: m^k-1 for multistage}, \ref{lem: z^k new update}, and \ref{lem: difference for multistage} in App. \ref{App: generalizations for multistage}). In Sec. \ref{sec: Multistage SGDM theory}. we shall see that with \eqref{equ: principle} and appropriately chosen $\{c_i\}_{i=1}^{\infty}$, $L^k$ in \eqref{equ: lyapunov} also defines a Lyapunov function in the multistage setting, which in turn leads to the convergence of Multistage SGDM.

\section{Convergence of SGDM}
\label{sec: SGDM theory}

In this section, we proceed to establish the convergence of SGDM described in \eqref{equ: momentum}. First, by following the idea presented in Sec. \ref{sec: key ingredients}, we can show that $L^k$ defined in \eqref{equ: lyapunov} is a Lyapunov function. 


%
%
%
%
%
%
%
%
%
%
%
%
%
%
%
%
%
%
%
%
%
%
%
%
%
%
%
%


\begin{proposition}
\label{prop: L^k}
Let Assumption \ref{assump: standard assumption} hold. In \eqref{equ: momentum}, let $\alpha\leq\frac{1-\beta}{2\sqrt{2}L\sqrt{\beta+\beta^2}}$. Let $\{c_i\}_{i=1}^{\infty}$ in \eqref{equ: lyapunov} be defined by
\begin{align*}
c_1 &= \frac{\frac{\beta+\beta^2}{(1-\beta)^3}L^3\alpha^2}{1-4\alpha^2\frac{\beta+\beta^2}{(1-\beta)^2}L^2}, \quad\quad c_{i+1}=c_i - \left(4c_1\alpha^2+\frac{L\alpha^2}{1-\beta}\right)\beta^i(i+\frac{\beta}{1-\beta})L^2 \quad \text{for all $i\geq 1$.}
\end{align*}
Then, $c_i>0$ for all $i\geq 1$, and 
\begin{align}
\label{equ: L^k nonconvex final}
\begin{split}
\E[L^{k+1}-L^k] &\leq \left(-\alpha+\frac{3-\beta+\beta^2}{2(1-\beta)}L\alpha^2+4c_1\alpha^2\right)\E[\|g^k\|^2]\\
&\quad +\left(\frac{\beta^2}{2(1+\beta)}L\alpha^2\sigma^2+\frac{1}{2}L\alpha^2\sigma^2+2c_1\frac{1-\beta}{1+\beta}\alpha^2\sigma^2\right).
\end{split}
\end{align}
\end{proposition}

By telescoping \eqref{equ: L^k nonconvex final}, we obtain the stationary convergence of SGDM under nonconvex settings.
\begin{theorem}
\label{thm: nonconvex constant}
Let Assumption \ref{assump: standard assumption} hold. In \eqref{equ: momentum}, let $\alpha\leq \alpha\leq \min\{\frac{1-\beta}{L(4-\beta+\beta^2)}, \frac{1-\beta}{2\sqrt{2}L\sqrt{\beta+\beta^2}}\}$. Then,
\begin{align}
\label{equ: SGDM bound}
\begin{split}
\frac{1}{k}\sum_{i=1}^k\E[\|g^i\|^2]
&\leq \frac{2\left(f(x^1)-f^*\right)}{k\alpha}+\left(\frac{\beta+5\beta^2}{8(1+\beta)}+1\right)L\alpha\sigma^2=\mathcal{O}\left(\frac{f(x^1)-f^*}{k\alpha}+L\alpha\sigma^2\right).
\end{split}
\end{align}
\end{theorem}

Now let us turn to the strongly convex setting, for which we have
\begin{proposition}
\label{prop: L^k scvx}
Let Assumption \ref{assump: standard assumption} hold. Assume in addition that $f$ is $\mu-$strongly convex. In \eqref{equ: momentum}, let $\alpha\leq \min\{\frac{1-\beta}{5 L},   \frac{1-\beta}{L\left(3-\beta+2\beta^2+\frac{48\sqrt{\beta}}{25}\frac{2L+18\mu}{L}\right)}\}$. Then, there exists positive constants $c_i$ for \eqref{equ: lyapunov} such that for all $k\geq k_0\coloneqq \lfloor \frac{\log 0.5}{\log \beta}\rfloor$, we have 
\begin{align*}
\E[L^{k+1}-L^k] & \leq -\frac{\alpha\mu}{1+\frac{8\mu}{L}}\E[L^k]+ (\frac{1+\beta+\beta^2}{2(1+\beta)}L+\frac{1-\beta}{1+\beta}2c_1)\alpha^2\sigma^2 +\frac{{\beta^2}+\frac{L\alpha}{2}\frac{\beta^2}{1-\beta}}{(1+\frac{8\mu}{L})(1+\beta)}2\mu\alpha^2\sigma^2.
\end{align*}
\end{proposition}
The choices of $\{c_i\}_{i=1}^{\infty}$ is similar to those of Proposition \ref{prop: L^k} and can be found in App. \ref{app: l^k scvx}. With Proposition \ref{prop: L^k scvx}, we immediately have
\begin{theorem}
\label{thm: scvx constant}
Let Assumption \ref{assump: standard assumption} hold and assume in addition that $f$ is $\mu-$strongly convex. Under the same settings as in Proposition \ref{prop: L^k scvx}, for all $k\geq k_0= \lfloor \frac{\log 0.5}{\log \beta}\rfloor$ we have 
\begin{align*}
\E[f(z^k)-f^*]&\leq \left(1-\frac{\alpha\mu}{1+\frac{8\mu}{L}}\right)^{k-k_0}\E[L^{k_0}]+\left(1+\frac{8\mu}{L}\right)\frac{1+\beta+\beta^2}{2(1+\beta)}\frac{L}{\mu}\alpha\sigma^2\\
&\,\,\,+ \left(1+\frac{8\mu}{L}\right)\left( \frac{1}{1+\beta}\frac{12\sqrt{\beta}}{25}\frac{2L+18\mu}{\mu}\alpha\sigma^2  + \frac{{\beta^2}+\frac{L\alpha}{10}{\beta^2}}{1+\frac{8\mu}{L}}\frac{2}{1+\beta}\alpha\sigma^2\right)\\
& = \mathcal{O}\left((1-\alpha\mu)^{k} + \frac{L}{\mu}\alpha\sigma^2\right).
\end{align*}
\end{theorem}

\begin{corollary}
\label{coro: to x^k}
Let Assumption \ref{assump: standard assumption} hold and assume in addition that $f$ is $\mu-$strongly convex. Under the same settings as in Proposition \ref{prop: L^k scvx}, for all $k\geq k_0= \lfloor \frac{\log 0.5}{\log \beta}\rfloor$ we have 
\[
\EE[f(x^k) - f^*]= \mathcal{O}\left(r^{k} + \frac{L}{\mu}\alpha\sigma^2\right),
\]
where $r = \max\{1-\alpha\mu, \beta\}$.
\end{corollary}
\begin{remark}
\begin{enumerate}
\item Under nonconvex settings, the classical convergence bound of SGD is \\$\mathcal{O}\left(\frac{f(x^1)-f^*}{k\alpha}+L\alpha\sigma^2\right)$ with $\alpha=\mathcal{O}(\frac{1}{L})$ (see, e.g., Theorem 4.8 of \cite{bottou2018optimization}). Therefore, Theorem \ref{thm: nonconvex constant} tells us that with $\alpha=\mathcal{O}(\frac{1-\beta}{L})$, SGDM achieves the same convergence bound as SGD.

\item In contrast, the radius of the  stationary distribution for SGDM in \cite{yu2019linear} and \cite{yan2018unified} is $\mathcal{O}(\frac{\alpha\sigma^2}{1-\beta})$, and the latter one also assumes that $\nabla f$ is uniformly bounded.  
\item In Theorem \ref{thm: scvx constant} and Corollary \ref{coro: to x^k}, the convergence bounds hold for $k\geq k_0 = \lfloor \frac{\log 0.5}{\log \beta}\rfloor$, where $k_0$ is a mild constant\footnotemark[1].
\item The convergence bound of SGD under strong convexity is $\mathcal{O}\left((1-\alpha\mu)^{k} + \frac{L}{\mu}\alpha\sigma^2\right)$ (see, e.g, Theorem 4.6 of \cite{bottou2018optimization}), our result for SGDM in Corollary \ref{coro: to x^k} recovers this when $\beta = 0$.
\end{enumerate}
\end{remark}
\footnotetext[1]{For example, we have $k_0 = 6$ for the popular choice $\beta=0.9$.}
\section{Convergence of Multistage SGDM}
\label{sec: Multistage SGDM theory}

In this section, we switch to the Multistage SGDM (Algorithm \ref{alg: Multistage SGDM}).

Let us first show that when the \eqref{equ: principle} is applied, we can define the constants $c_i$ properly so that \eqref{equ: lyapunov} still produces a Lyapunov function. 
\begin{proposition}
\label{prop: L^k for Multistage}
Let Assumption \ref{assump: standard assumption} hold. In Algorithm \ref{alg: Multistage SGDM}, let the parameters satisfy  \eqref{equ: principle} with $A_1 =\frac{1}{24\sqrt{2}L}$.
In addition, let 
\begin{align*}
\frac{1-\beta_1}{\beta_1}&\leq 12\frac{1-\beta_{n}}{\sqrt{\beta_{n}+\beta^2_{n}}}, \quad\quad
c_1 = \frac{\frac{\alpha_1^2}{1-\beta_1}\frac{\beta_n+\beta_n^2}{(1-\beta_n)^2}L^3}{1-4\alpha_1^2\frac{\beta_n+\beta_n^2}{(1-\beta_n)^2}L^2},
\end{align*}
and for any $i\geq 1$, let
\begin{align*}
c_{i+1}&= c_i -\bigg(4c_1\alpha^2_1+ L\frac{\alpha_1^2}{1-\beta_1}\bigg)\beta_n^i(i+\frac{\beta_n}{1-\beta_n})L^2.
\end{align*}
Then, we have $c_i>0$ for any $i\geq 1$. Furthermore, with $z^k$ defined in \eqref{equ: z^k new}, for any $k\geq 1$, we have
\begin{align*}
&\E[L^{k+1}-L^k]\\
&\leq \bigg(-\alpha(k)+\frac{3-\beta(k)+2\beta^2(k)}{2(1-\beta(k))}L\alpha^2(k) +4c_1\alpha^2(k)\bigg)\E[\|g^k\|^2]\\
&\quad +\bigg({\beta^2(k)}L\alpha^{2}(k)12 \frac{\beta_1}{\sqrt{\beta_n+\beta^2_n}}\sigma^2+\frac{1}{2}L\alpha^2(k)\sigma^2+4c_1(1-\beta_1)\alpha^2(k)\sigma^2\bigg).
\end{align*}
where $\alpha(k), \beta(k)$ are the stepsize and momentum weight applied at $k$th iteration, respectively.
\end{proposition}

\begin{theorem}
\label{thm: nonconvex Multistage}
Let Assumption \ref{assump: standard assumption} hold. Under the same settings as in Proposition \ref{prop: L^k for Multistage}, let $\beta_1\geq \frac{1}{2}$ and let $A_2$ be large enough such that
\[
\beta_i^{2T_i}\leq \frac{1}{2} \quad \text{for}\quad i=1,2,...n.
\]
Then, we have
\begin{align}
\label{equ: multistage bound}
\begin{split}
\frac{1}{n}\sum_{l=1}^n\frac{1}{T_l}\sum_{i=T_1+.. +T_{l-1}+1}^{T_1+..  +T_l} \E[\|g^i\|^2]&\leq \frac{2(f(x^1)-f^*)}{n A_2} +\frac{1}{n}\sum_{l=1}^n \left(24 \beta_{l}^2\frac{\beta_1}{\sqrt{\beta_{n}+\beta_{n}^2}} L +3L\right)\alpha_l \sigma^2 \\
&= \mathcal{O}\left(\frac{f(x^1)-f^*}{n A_2}+\frac{1}{n}\sum_{l=1}^n L \alpha_l\sigma^2\right).
\end{split}
\end{align}
\end{theorem}


\begin{remark}
\begin{enumerate}
    \item On the left hand side of \eqref{equ: multistage bound}, we have the average of the averaged squared gradient norm of $n$ stages.
    \item On the right hand side of \eqref{equ: multistage bound}, the first term dominates at initial stages, we can apply large $\alpha_i$ for these stages to accelerate initial convergence, and use smaller $\alpha_i$ for later stages so that the size of stationary distribution is small. In contrast, (static) SGDM need to use a small stepsize $\alpha$ to make the size of stationary distribution with small.
    \item It is unclear whether the overall iteration complexity of Multistage SGDM is better than SGDM or not. Numerically, we do observe that Multistage SGDM is faster. We leave the possible improved analysis of Multistage SGDM to future work.
\end{enumerate}
\end{remark}


\section{Experiments}

In this section, we verify our theoretical claims by numerical experiments. For each combination of algorithm and training task, training is performed with $3$ random seeds $1, 2, 3$.  Unless otherwise stated, we report the average of losses of the past $m$ batches, where $m$ is the number of batches for the whole dataset. Additional implementation details can be found in App. \ref{app: details and other experiments}.

\subsection{Logistic regression}

\textbf{Setup.} The MNIST dataset consists of $n = 60000$ labeled examples of $28\times 28 $ gray-scale images of handwritten digits in $K = 10$ classes $0, 1, \dots, 9$. For all algorithms, we use batch size $s = 64$ (and hence number batches per epoch is $m=1874$), number of epochs $T = 50$. The regularization parameter is $\lambda = 5 \times 10^{-4}$.

\textbf{The effect of $\alpha$ in (static) SGDM.} By Theorem \ref{thm: scvx constant} we know that, with a fixed $\beta$, a larger $\alpha$ leads to faster loss decrease to the stationary distribution. However, the size of the stationary distribution is also larger. This is well illustrated in Figure \ref{fig:logistic-0.9}. For example, $\alpha = 1.0$ and $\alpha=0.5$ make losses decrease more rapidly than $\alpha=0.1$.
During later iterations, $\alpha=0.1$ leads to a lower final loss.



\begin{figure}[ht]
\vskip 0.0in
\begin{center}
\includegraphics[width=0.8\columnwidth]{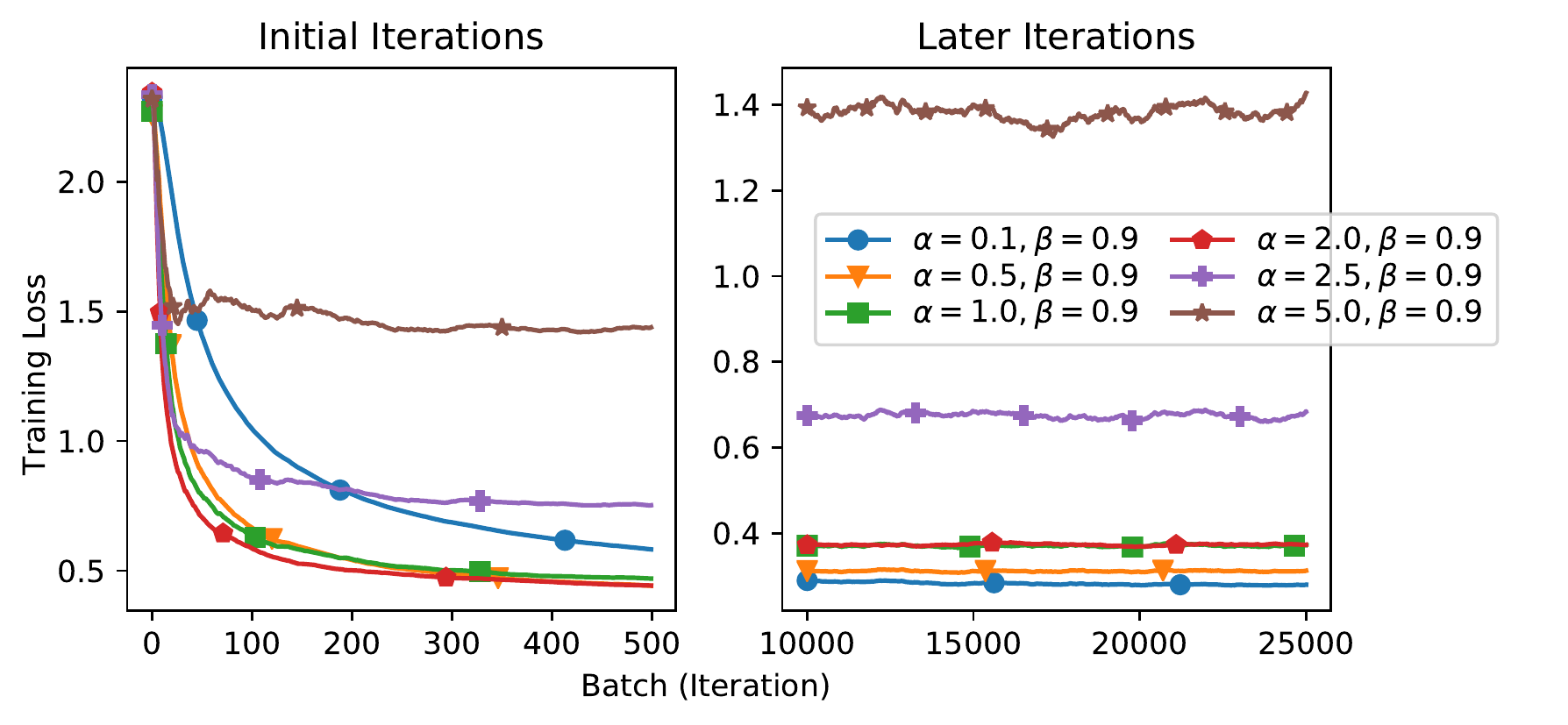}
\vskip -0.1in
\caption{Logistic Regression on the MNIST Dataset using SGDM with fixed $(\alpha, \beta)$}
\label{fig:logistic-0.9}
\end{center}
\vskip -0.2in
\end{figure}

\textbf{Multistage SGDM.} We take $3$ stages for Multistage SGDM. The parameters are chosen according to \eqref{equ: principle}: $T_1 = 3, T_2 = 6, T_3 = 21$, $\alpha_i = A_2/T_i$, $\beta_i = A_1/(c_2+\alpha_i)$, where $A_2 = 2.0$ and $A_1 =1.0$.\footnotemark[1] We compare Multistage SGDM with SGDM with $(\alpha, \beta)=(0.66, 0.9)$ and  $(\alpha, \beta)=(0.095, 0.9)$, where $0.66$, $0.095$ are the stepsizes of the first and last stage of Multistage SGDM, respectively. The training losses of initial and later iterations are shown in Figure \ref{fig:logistic-mssgdm_vs_fix_beta}. 

\footnotetext[1]{Here, $A_1$ is not set by its theoretical value $\frac{1}{24L}$, since the dataset is very large and the gradient Lipschitz constant $L$ cannot be computed easily.}

We can see that SGDM with $(\alpha, \beta)=(0.66, 0.9)$ converges faster initially, but has a higher final loss; while SGDM with $(\alpha, \beta)=(0.095, 0.9)$ behaves the other way. Multistage SGDM takes the advantage of both, as predicted by Theorem \ref{thm: nonconvex Multistage}. The performances of SGDM and Vanilla SGD with the same stepsize are similar.
\begin{figure}[ht]
\hspace*{0.45in}
\includegraphics[width=0.7\columnwidth]{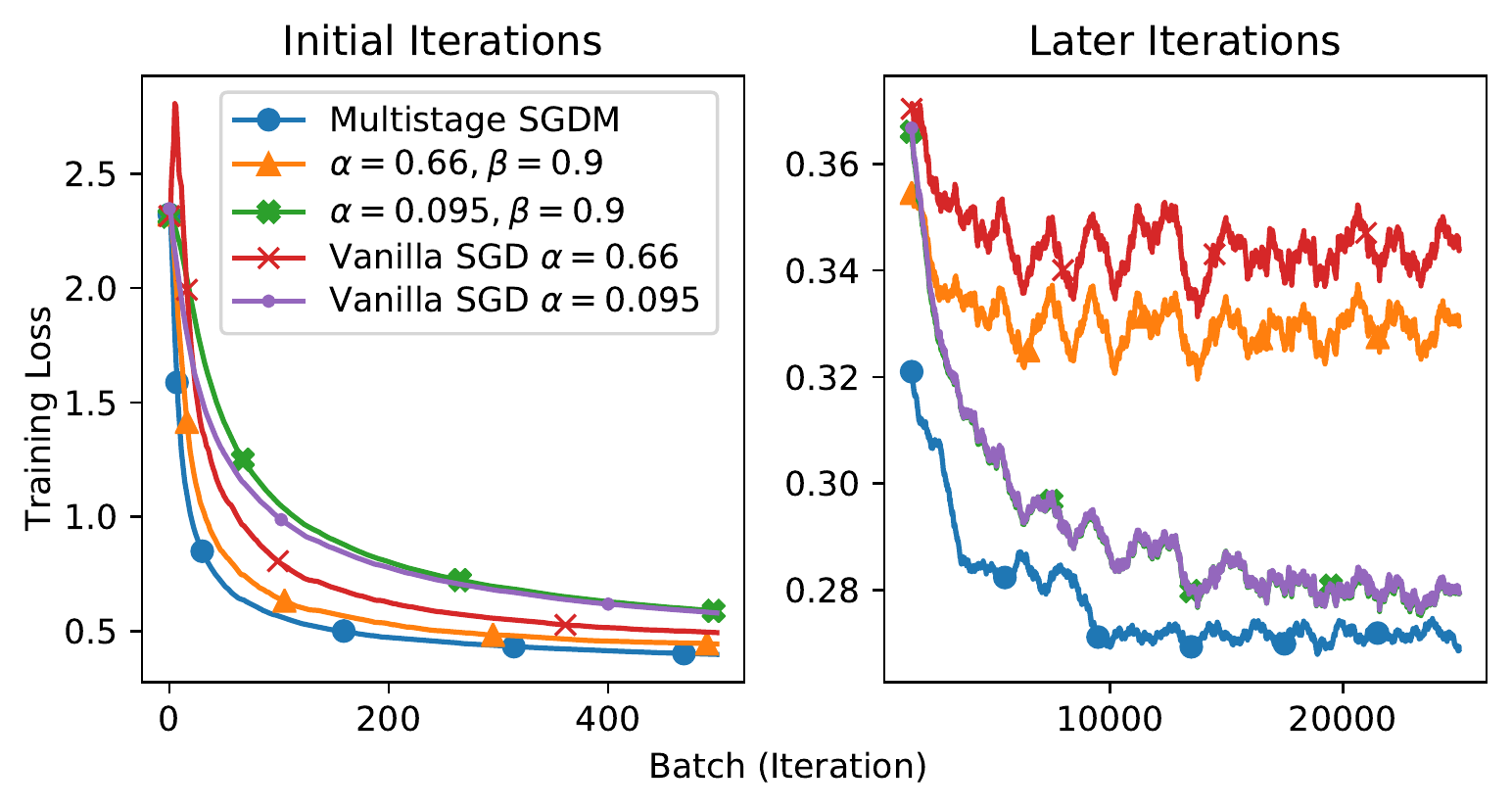}
\caption{Logistic Regression on the MNIST using Multistage SGDM and SGDM with fixed $\beta$}
\label{fig:logistic-mssgdm_vs_fix_beta}
\vskip -0.2in
\end{figure}

\subsection{Image classification}\label{subsec:image-class}


For the task of training ResNet18 on the CIFAR-10 dataset, we compare Multistage SGDM, a baseline SGDM, and YellowFin \cite{zhang2017yellowfin}, an automatic momentum tuner based on heuristics from optimizing strongly convex quadratics. The initial learning rate of YellowFin is set to $0.1$,\footnote{We have experimented with initial learning rates $0.001$ (default), $0.01$, $0.1$ and $0.5$, each repeated $3$ times; we found that the choice $0.1$ is the best in terms of the final training loss.} and other parameters are set as their default values.  All algorithms are run for $T = 50$ epochs and the batch size is fixed as $s = 128$. 

For Multistage SGDM, the parameters choices are governed by \eqref{equ: principle}: the stage lengths are $T_1 = 5$, $T_2 = 10$, and $T_3 = 35$. Take $A_1 = 1.0$, $A_2 = 2.0$, set the per-stage stepsizes and momentum weights as $\alpha_i = A_2/T_i$ and $\beta_i = A_1 / (A_1 + \alpha_i)$, for stages $i = 1, 2, 3$. For the baseline SGDM, the stepsize schedule of Multistage SGDM is applied, but with a fixed momentum $\beta = 0.9$.

In Figure \ref{fig:resnet-cifar}, we present training losses and end-of-epoch validation accuracy of the tested algorithms. We can see that Multistage SGDM performs the best. Baseline SGDM is slightly worse, possibly because of its fixed momentum weight.

\vskip -0.1in
\begin{figure}[H]
\hspace*{0.32in} \includegraphics[width=0.72\columnwidth]{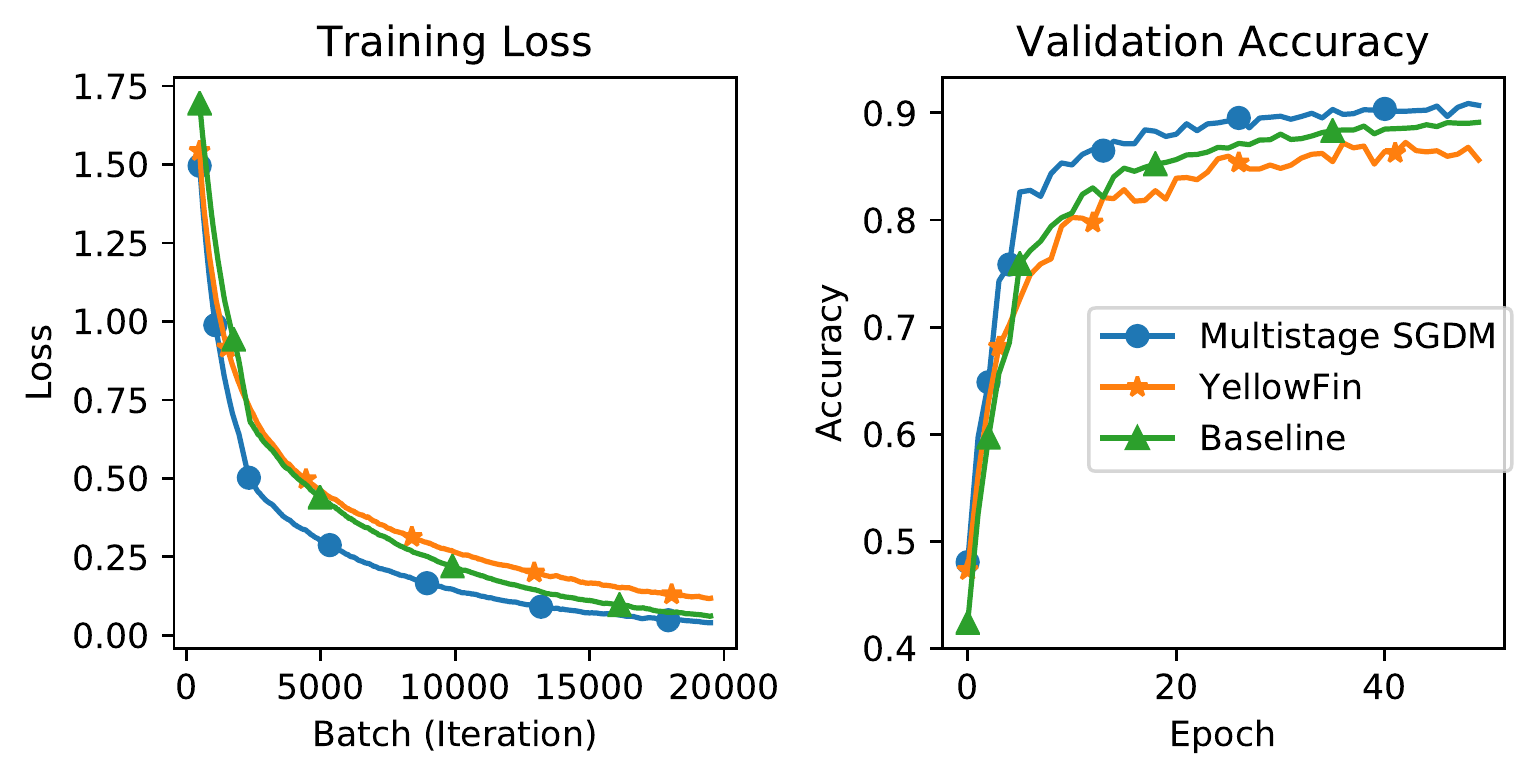}
\caption{Training ResNet18 on CIFAR10}
\label{fig:resnet-cifar}
\vskip -0.2in
\end{figure}

\section{Summary and Future Directions}
In this work, we provide new theoretical insights into the convergence behavior of SGDM and Multistage SGDM. For SGDM, we show that it is as fast as plain SGD in both nonconvex and strongly convex settings. For the widely adopted multistage SGDM, we establish its convergence  and show the advantage of stagewise training. 

There are still open problems to be addressed. For example, (a) Is it possible to show that SGDM converges faster than SGD for special objectives such as quadratic ones? (b) Are there more efficient parameter choices than \eqref{equ: principle} that guarantee even faster convergence?



\newpage 
 

%
\bibliographystyle{plain}
\bibliography{references}

\newpage 
\appendix

\section{Proof of Preliminary Lemmas}
\label{app: proof of lemmas}

\subsection{Proof of Lemma \ref{lem: m^k}}
Since $m^k = (1-\beta)\sum_{i=1}^k \beta^{k-i}\tg^i$, we have
\[
\E\left[\left\|m^k - (1-\beta)\sum_{i=1}^k \beta^{k-i}g^i\right\|^2\right]
=(1-\beta)^2\E\left[\left\|\sum_{i=1}^k\beta^{k-i}(\tilde{g}^i-g^i)\right\|^2\right].
\]
Moreover, since $\zeta^1,\zeta^2,...,\zeta^k$ are independent random variables (item 3 of Assumption \ref{assump: standard assumption}), we can write the total expectation as $\EE = \EE_{\zeta^1}\EE_{\zeta^2}...\EE_{\zeta^{k}}$, and therefore
\begin{align*}
& \E\left[\left\|m^k - (1-\beta)\sum_{i=1}^k \beta^{k-i}g^i\right\|^2\right]\\
&= (1-\beta)^2\EE_{\zeta^1}\EE_{\zeta^2}...\EE_{\zeta^{k}}\left[\left\|\sum_{i=1}^k\beta^{k-i}(\tilde{g}^i-g^i)\right\|^2\right]\\
& = (1-\beta)^2\EE_{\zeta^1}\EE_{\zeta^2}...\EE_{\zeta^{k}}\left[\sum_{i=1}^k\sum_{j=1}^k\langle\beta^{k-i}(\tilde{g}^i-g^i), \beta^{k-j}(\tilde{g}^j-g^j)\rangle\right].
\end{align*}
By applying $\EE_{\zeta^i}[\tg^i] = g^i$ (item 2 in Assumption \ref{assump: standard assumption}), we further have for any $i>j$ that
\begin{align*}
&\EE_{\zeta^1}\EE_{\zeta^2}...\EE_{\zeta^{k}}\left[\langle\tilde{g}^i-\EE_{\zeta^{i}}[\tilde{g}^i], \tilde{g}^j-\EE_{\zeta^{j}}[\tilde{g}^j]\rangle\right]\\
& = \EE_{\zeta^1}\EE_{\zeta^2}...\EE_{\zeta^{i}}\left[\langle\tilde{g}^i-\EE_{\zeta^{i}}[\tilde{g}^i], \tilde{g}^j-\EE_{\zeta^{j}}[\tilde{g}^j]\rangle\right]\\
& = \EE_{\zeta^1}\EE_{\zeta^2}...\EE_{\zeta^{i-1}}\left[\langle\EE_{\zeta^i}[\tilde{g}^i]-\EE_{\zeta^{i}}[\tilde{g}^i], \tilde{g}^j-\EE_{\zeta^{j}}[\tilde{g}^j]\rangle\right]\\
&=0.
\end{align*}.

It is straightforward to see that the same conclusion holds for $i<j$.

Finally, we know from the item 4 in Assumption \ref{assump: standard assumption} that
\begin{align*}
&\E\left[\left\|m^k - (1-\beta)\sum_{i=1}^k \beta^{k-i}g^i\right\|^2\right]\\
& = (1-\beta)^2\EE_{\zeta^1}\EE_{\zeta^2}...\EE_{\zeta^{k}}\left[\sum_{i=1}^k \beta^{2(k-i)} \|\tilde{g}^i-\EE_{\zeta^{i}}[\tilde{g}^i]\|^2\right]\\
&\leq \frac{1-\beta}{1+\beta}(1-\beta^{2k})\sigma^2.
\end{align*}

\subsection{Proof of Lemma \ref{lem: difference}}

We have 
\begin{align*} 
&\E\left[\left\| \frac{1-\beta}{1-\beta^{k}}\sum_{i=1}^{k} \beta^{k-i} g^i-g^k \right\|^2\right]\\
&=\left(\frac{1-\beta}{1-\beta^k}\right)^2\sum_{i,j=1}^k\E[\langle \beta^{k-i}(g^k-g^i), \beta^{k-j}(g^k-g^j)\rangle]\\
&\leq \left(\frac{1-\beta}{1-\beta^k}\right)^2\sum_{i,j=1}^k\beta^{2k-i-j}(\frac{1}{2}\E[\|g^k-g^i\|^2]+\frac{1}{2}\E[\|g^k-g^j\|^2])\\
&=\left(\frac{1-\beta}{1-\beta^k}\right)^2\sum_{i=1}^k\left(\sum_{j=1}^k\beta^{2k-i-j}\right)\frac{1}{2}\E[\|g^k-g^j\|^2]\\
&\quad + \left(\frac{1-\beta}{1-\beta^k}\right)^2\sum_{j=1}^k\left(\sum_{i=1}^k\beta^{2k-i-j}\right)\frac{1}{2}\E[\|g^k-g^i\|^2]\\
&= \left(\frac{1-\beta}{1-\beta^k}\right)^2 \sum_{i=1}^k\frac{\beta^{k-i}(1-\beta^k)}{1-\beta}\E[\|g^k-g^i\|^2]\\
&=\frac{1-\beta}{1-\beta^k}\sum_{i=1}^k\beta^{k-i}\E[\|g^k-g^i\|^2],
\end{align*}
where we have applied Cauchy-Schwarz in the first inequality.

By applying triangle inequality and the smoothness of $f$ (item 1 in Assumption \ref{assump: standard assumption}), we further have 
\begin{align*}
&\E\left[\left\| \frac{1-\beta}{1-\beta^{k}}\sum_{i=1}^{k} \beta^{k-i} g^i-g^k \right\|^2\right]\\
&\leq \frac{1-\beta}{1-\beta^k}\sum_{i=1}^k\beta^{k-i}(k-i)\sum_{j=i}^{k-1}\E[\|g^{j+1}-g^j\|^2]\\
&\leq \frac{1-\beta}{1-\beta^k}\sum_{i=1}^k\beta^{k-i}(k-i)\sum_{j=i}^{k-1}L^2\E[\|x^{j+1}-x^j\|^2]\\
&= \frac{1-\beta}{1-\beta^k}\sum_{j=1}^{k-1}\left(\sum_{i=1}^j\beta^{k-i}(k-i)\right)L^2\E[\|x^{j+1}-x^j\|^2].
\end{align*}
Therefore, by defining $a'_{k,j}=\frac{1-\beta}{1-\beta^k}L^2\sum_{i=1}^j\beta^{k-i}(k-i)$, we get 
\begin{align}
\label{equ: difference by a' k j}
\begin{split}
&\E\left[\left\| \frac{1-\beta}{1-\beta^{k}}\sum_{i=1}^{k} \beta^{k-i} g^i-g^k \right\|^2\right]\leq \sum_{j=1}^{k-1} a'_{k,j}\E[\|x^{j+1}-x^j\|^2].
\end{split}
\end{align}
Furthermore, $a'_{k,j}$ can be calculated as
\begin{align}
\label{equ: a' k j formula}
\begin{split}
a'_{k,j}&= \frac{L^2\beta^k}{1-\beta^k}\left(-(k-1)-\frac{1}{1-\beta}\right) +\frac{L^2\beta^{k-j}}{1-\beta^k}\left(k-j+\frac{\beta}{1-\beta}\right).
\end{split}
\end{align}
Notice that 
\begin{align}
\label{equ: a' k j upper bound}
a'_{k,j}<a_{k,j}\coloneqq \frac{L^2\beta^{k-j}}{1-\beta^k}\left(k-j+\frac{\beta}{1-\beta}\right).
\end{align}
Combining this with \eqref{equ: difference by a' k j}, we finally arrive at
\begin{align*}
\E\left[\left\| \frac{1-\beta}{1-\beta^{k}}\sum_{i=1}^{k} \beta^{k-i} g^i-g^k \right\|^2\right]\leq \sum_{i=1}^{k-1}a_{k,i}\E[\|x^{i+1}-x^i\|^2],
\end{align*}
where 
\begin{align*}
a_{k, i}=\frac{L^2\beta^{k-i}}{1-\beta^k}\left(k-i+\frac{\beta}{1-\beta}\right).
\end{align*}

\subsection{Proof of Lemma \ref{lem: z^k update}}

Let us consider the cases of $k=1$ and $k\geq 2$ separately. 

For $k=1$, we have
\begin{align*}
z^2-z^1 &= \frac{1}{1-\beta}x^2-\frac{\beta}{1-\beta}x^1-x^1=\frac{1}{1-\beta}(x^2-x^1)=-\alpha \tg^1.
\end{align*}
And for $k\geq 2$, we have
\begin{align*}
z^{k+1}-z^k&=\frac{1}{1-\beta}(x^{k+1}-x^k)-\frac{\beta}{1-\beta}(x^k-x^{k-1})\\
&=\frac{1}{1-\beta}(-\alpha m^k)-\frac{\beta}{1-\beta}(-\alpha m^{k-1})\\
&=\frac{1}{1-\beta}(-\alpha m^k+\alpha\beta m^{k-1})\\
&=-\alpha\tg^k.
\end{align*}

%
%
%
%
%
%
%
%
%
%

%
%
%

%
%
%
%
%
%
%
%
%
%
%
%
%
%


\section{Main Theory for SGDM}
\label{app: proof for SGDM}

\subsection{Objective descent}
In order to prove Proposition \ref{prop: L^k}, let us first show an auxiliary result.
\begin{proposition}
\label{prop: f(z^k)}
Take Assumption \ref{assump: standard assumption}. Then, for $z^k$ defined in \eqref{equ: z^k}, we have
\begin{align}
\label{equ: f(z^k)}
\begin{split}
\E[f(z^{k+1})]&\leq \E[f(z^k)] +(-\alpha+\frac{1+\beta^2}{1-\beta}L\alpha^2+\frac{1}{2}L\alpha^2)\E[\|g^k\|^2]\\
&\quad +(\frac{\beta^2}{2(1+\beta)}+\frac{1}{2})L\alpha^2\sigma^2 +\frac{\beta^2(1-\beta^k)^2L\alpha^2}{1-\beta}\E\left[\left\| \frac{1-\beta}{1-\beta^{k}}\sum_{i=1}^{k} \beta^{k-i} g^i-g^k \right\|^2\right].
\end{split}
\end{align}
\end{proposition}

The smoothness of $f$ yields
\begin{align}
\label{equ: f smoothness}
\begin{split}
\EE_{\zeta^k}[f(z^{k+1})]&\leq  f(z^k) + \EE_{\zeta^k} [\langle \nabla f(z^k), z^{k+1}-z^k \rangle] +\frac{L}{2}\EE_{\zeta^k}[\|z^{k+1}-z^k\|^2]\\
&= f(z^k) + \EE_{\zeta^k} [\langle \nabla f(z^k), -\alpha \tg^k \rangle] +\frac{L\alpha^2}{2}\EE_{\zeta^k}[\|\tg^k\|^2],
\end{split}
\end{align}
where we have applied Lemma \ref{lem: z^k update} in the second step.

For the inner product term, we can take full expectation $\EE = \EE_{\zeta^1}...\EE_{\zeta^k}$ to get
\begin{align*}
&\E [\langle \nabla f(z^k), -\alpha \tg^k \rangle] = \E [\langle \nabla f(z^k), -\alpha g^k \rangle],
\end{align*}
which follows from the fact that $z^k$ is determined by the previous $k-1$ random samples $\zeta^1, \zeta^2,...\zeta^{k-1}$, which is independent of $\zeta^k$, and $\E_{\zeta^k}[\tg^k]=g^k$.

So, we can bound
\begin{align*}
\E [\langle \nabla f(z^k), -\alpha \tg^k \rangle] &=\E [\langle \nabla f(z^k)-g^k, -\alpha g^k \rangle] - \alpha \E[\|g^k\|^2]\\
&\leq \alpha \frac{\rho_{0}}{2}L^2\E[\|z^k-x^k\|^2]+\alpha\frac{1}{2\rho_{0}}\E[\|g^k\|^2]-\alpha \E[\|g^k\|^2],
\end{align*}
where $\rho_{0}>0$ can be any positive constant (to be determined later).

Combining \eqref{equ: f smoothness} and the last inequality, we arrive at
\begin{align*}
\E[f(z^{k+1})]&\leq \E[f(z^k)]+\alpha\frac{\rho_0}{2}L^2\E[\|z^k-x^k\|^2]\\
&\quad+(\alpha\frac{1}{2\rho_0}-\alpha)\E[\|g^k\|^2]+\frac{L\alpha^2}{2}\E[\|\tg^k\|^2].
\end{align*}
Since $z^k = x^k$ when $k=1$ and $z^k=\frac{1}{1-\beta}x^k-\frac{\beta}{1-\beta}x^{k-1}$ when $k\geq 2$, it can be verified that $z^k-x^k=-\frac{\beta}{1-\beta}\alpha m^{k-1}$. Consequently, 
\begin{align}
\label{equ: f(z^k) intermediate}
\begin{split}
    \E[f(z^{k+1})]&\leq \E[f(z^k)]+\alpha^3\frac{\rho_0}{2}L^2(\frac{\beta}{1-\beta})^2\E[\|m^{k-1}\|^2]\\
&\quad+(\alpha\frac{1}{2\rho_0}-\alpha)\E[\|g^k\|^2]+\frac{L\alpha^2}{2}\E[\|\tg^k\|^2].
\end{split}
\end{align}
On the other hand, from Lemma \ref{lem: m^k} we know that 
\begin{align}
\label{equ: useful inequalities 1}
\begin{split}
\E[\|m^{k-1}\|^2]& \leq 2\E[\|m^{k-1}-(1-\beta)\sum_{i=1}^{k-1} \beta^{k-1-i} g^i\|^2] + 2\E[\|(1-\beta)\sum_{i=1}^{k-1} \beta^{k-1-i} g^i\|^2]\\
&\leq 2\frac{1-\beta}{1+\beta}\sigma^2 + 2 \E[\|(1-\beta)\sum_{i=1}^{k-1} \beta^{k-1-i} g^i\|^2]\\
\E[\|\frac{1-\beta}{1-\beta^{k-1}}\sum_{i=1}^{k-1} \beta^{k-1-i} g^i\|^2]
&\leq 2\E[\|g^k\|^2]+2\E[\|\frac{1-\beta}{1-\beta^{k-1}}\sum_{i=1}^{k-1} \beta^{k-1-i} g^i-g^k\|^2],\\
\E[\|\tg^k\|^2]&\leq \sigma^2 + \E[\|g^k\|^2].
\end{split}
\end{align}
Putting these into \eqref{equ: f(z^k) intermediate}, we arrive at
\begin{align*}
\E[f(z^{k+1})]
&\leq \E[f(z^k)] +\bigg(-\alpha+\alpha\frac{1}{2\rho_0}+2\alpha^3\rho_0L^2(\frac{\beta}{1-\beta})^2(1-\beta^{k-1})^2+\frac{L\alpha^2}{2}\bigg)\E[\|g^k\|^2]\\
&\quad+\left(\alpha^3{\rho_0}L^2(\frac{\beta}{1-\beta})^2\frac{1-\beta}{1+\beta}\sigma^2+\frac{L\alpha^2}{2}\sigma^2\right)\\
&\quad+2\alpha^3\rho_0 L^2(\frac{\beta}{1-\beta})^2(1-\beta^{k-1})^2\E[
\|\frac{1-\beta}{1-\beta^{k-1}}\sum_{i=1}^{k-1} \beta^{k-1-i} g^i-g^k\|^2].
\end{align*}
Substituting
\begin{align*}
\E[\|\frac{1-\beta}{1-\beta^{k}}\sum_{i=1}^{k} \beta^{k-i} g^i-g^k\|^2]
&=\E[\|\frac{1}{1-\beta^k}\beta(1-\beta)\sum_{i=1}^{k-1} \beta^{k-1-i} g^i-\frac{1-\beta^{k-1}}{1-\beta^k}\beta g^k\|^2]\\
&=\beta^2(\frac{1-\beta^{k-1}}{1-\beta^k})^2\E [\|\frac{1-\beta}{1-\beta^{k-1}}\sum_{i=1}^{k-1} \beta^{k-1-i} g^i-g^k\|^2]
\end{align*}
into the last inequality produces
\begin{align}
\label{equ: f(z^k) intermediate 1}
\begin{split}
\E[f(z^{k+1})]
&\leq \E[f(z^k)] +\bigg(-\alpha+\alpha\frac{1}{2\rho_0}+2\alpha^3\rho_0L^2(\frac{\beta}{1-\beta})^2(1-\beta^{k-1})^2+\frac{L\alpha^2}{2}\bigg)\E[\|g^k\|^2]\\
&\quad+\left(\alpha^3{\rho_0}L^2(\frac{\beta}{1-\beta})^2\frac{1-\beta}{1+\beta}\sigma^2+\frac{L\alpha^2}{2}\sigma^2\right)\\
&\quad +2\alpha^3\rho_0 L^2(\frac{1}{1-\beta})^2(1-\beta^{k})^2\E [\|\frac{1-\beta}{1-\beta^{k}}\sum_{i=1}^{k} \beta^{k-i} g^i-g^k\|^2].
\end{split}
\end{align}
Finally, using $1-\beta^{k-1}<1$ and $\rho_0=\frac{1-\beta}{2L\alpha}$ gives
\begin{align*}
\begin{split}
\E[f(z^{k+1})]&\leq \E[f(z^k)] +(-\alpha+\frac{1+\beta^2}{1-\beta}L\alpha^2+\frac{1}{2}L\alpha^2)\E[\|g^k\|^2]\\
&\quad +(\frac{\beta^2}{2(1+\beta)}+\frac{1}{2})L\alpha^2\sigma^2 +\frac{\beta^2(1-\beta^k)^2L\alpha^2}{1-\beta}\E\left[\left\| \frac{1-\beta}{1-\beta^{k}}\sum_{i=1}^{k} \beta^{k-i} g^i-g^k \right\|^2\right].
\end{split}
\end{align*}

\subsection{Proof of Proposition \ref{prop: L^k}}
Recall that $L^k$ is defined as
\[
L^k = f(z^k) - f^* +\sum_{i=1}^{k-1}c_i\|x^{k+1-i}-x^{k-i}\|^2,
\]
Therefore, by \eqref{equ: f(z^k) intermediate 1} we know that
\begin{align}
\label{equ: L^k intermediate}
\begin{split}
\E[L^{k+1}-L^k]
&\leq \bigg(-\alpha+\alpha\frac{1}{2\rho_0}+2\alpha^3\rho_0L^2(\frac{\beta}{1-\beta})^2+\frac{L\alpha^2}{2}\bigg)\E[\|g^k\|^2]\\
&\quad +\bigg(\alpha^3\rho_0L^2(\frac{\beta}{1-\beta})^2\frac{1-\beta}{1+\beta}\sigma^2+\frac{1}{2}L\alpha^2\sigma^2\bigg)\\
&\quad + \sum_{i=1}^{k-1}(c_{i+1}-c_i)\E[\|x^{k+1-i}-x^{k-i}\|^2]\\
&\quad +c_1\E[\|x^{k+1}-x^k\|^2]\\
&\quad +2\alpha^3\rho_0 L^2(\frac{1}{1-\beta})^2(1-\beta^{k})^2\E[\|\frac{1-\beta}{1-\beta^{k}}\sum_{i=1}^{k} \beta^{k-i} g^i-g^k\|^2],
\end{split}
\end{align}
where $\rho_0=\frac{1-\beta}{2L\alpha}$. 

To bound the $c_1\E[\|x^{k+1}-x^k\|^2]$ term, we need to following inequalities, which are obtained in a similar way as \eqref{equ: useful inequalities 1}.
\begin{align}
\label{equ: useful inequalities 2}
\begin{split}
\E[\|m^{k}\|^2]& \leq 2\E[\|m^{k}-(1-\beta)\sum_{i=1}^{k} \beta^{k-i} g^i\|^2] + 2\E[\|(1-\beta)\sum_{i=1}^{k} \beta^{k-i} g^i\|^2]\\
&\leq 2\frac{1-\beta}{1+\beta}\sigma^2 + 2 \E[\|(1-\beta)\sum_{i=1}^{k} \beta^{k-i} g^i\|^2]\\
\E[\|\frac{1-\beta}{1-\beta^{k}}\sum_{i=1}^{k} \beta^{k-i} g^i\|^2]
&\leq 2\E[\|g^k\|^2]+2\E[\|\frac{1-\beta}{1-\beta^{k}}\sum_{i=1}^{k} \beta^{k-i} g^i-g^k\|^2],\\
\E[\|\tg^k\|^2]&\leq \sigma^2 + \E[\|g^k\|^2].
\end{split}
\end{align}
Therefore, $c_1\E [\|x^{k+1}-x^k\|^2]$ can be bounded as
\begin{align*}
c_1\E[\|x^{k+1}-x^k\|^2]
&=c_1\alpha^2\E[\|m^k\|^2]\\
&\leq  c_1\alpha^2\left(2\frac{1-\beta}{1+\beta}\sigma^2+4\E[\|g^k\|^2](1-\beta^k)^2\right)\\
&\quad + 4 c_1\alpha^2  \E[\|\frac{1-\beta}{1-\beta^{k}}\sum_{i=1}^{k} \beta^{k-i} g^i-g^k\|^2]\\
&< c_1\alpha^2\left(2\frac{1-\beta}{1+\beta}\sigma^2+4\E[\|g^k\|^2]\right)\\
&\quad + 4 c_1\alpha^2(1-\beta^k)^2  \E[\|\frac{1-\beta}{1-\beta^{k}}\sum_{i=1}^{k} \beta^{k-i} g^i-g^k\|^2]
\end{align*}
Combine this with \eqref{equ: L^k intermediate}, we obtain
\begin{align}
\label{equ: L^k intermediate 1}
\begin{split}
&\E[L^{k+1}-L^k]\\
&\leq \left(-\alpha+\alpha\frac{1}{2\rho_0}+2\alpha^3\rho_0L^2(\frac{\beta}{1-\beta})^2+\frac{L\alpha^2}{2}+4c_1\alpha^2\right)\E[\|g^k\|^2]\\
&\quad +\left(\alpha^3\rho_0L^2(\frac{\beta}{1-\beta})^2\frac{1-\beta}{1+\beta}\sigma^2+\frac{1}{2}L\alpha^2\sigma^2+2c_1\frac{1-\beta}{1+\beta}\alpha^2\sigma^2\right)\\
&\quad + \sum_{i=1}^{k-1}(c_{i+1}-c_i)\E[\|x^{k+1-i}-x^{k-i}\|^2]\\
&\quad +4c_1\alpha^2(1-\beta^k)^2\E[\|\frac{1-\beta}{1-\beta^{k}}\sum_{i=1}^{k} \beta^{k-i} g^i-g^k\|^2]\\
&\quad +2\alpha^3\rho_0 L^2(\frac{1}{1-\beta})^2(1-\beta^{k})^2\E[\|\frac{1-\beta}{1-\beta^{k}}\sum_{i=1}^{k} \beta^{k-i} g^i-g^k\|^2].
\end{split}
\end{align}
In the rest of the proof, let us show that the sum of the last three terms in \eqref{equ: L^k intermediate 1} is non-positive.

First of all, by Lemma \ref{lem: difference} we know that
\begin{align*}
\E\left[\left\|\frac{1}{1-\beta^k}(1-\beta)\sum_{i=1}^k \beta^{k-i}g^i-g^k\right\|^2\right]\leq \sum_{i=1}^{k-1}a_{k,i}\E[\|x^{i+1}-x^i\|^2],
\end{align*}
where 
\begin{align*}
a_{k, i}=\frac{L^2\beta^{k-i}}{1-\beta^k}\left(k-i+\frac{\beta}{1-\beta}\right).
\end{align*}
Or equivalently, 
\begin{align*}
\E\left\|\frac{1}{1-\beta^k}(1-\beta)\sum_{i=1}^k \beta^{k-i}g^i-g^k\right\|^2\leq \sum_{i=1}^{k-1}a_{k,k-i}\E\|x^{k+1-i}-x^{k-i}\|^2,
\end{align*}
where 
\begin{align*}
a_{k, k-i}=\frac{L^2\beta^{i}}{1-\beta^k}\left(i+\frac{\beta}{1-\beta}\right).
\end{align*}
Therefore, in order to make the sum of the last three terms of \eqref{equ: L^k intermediate 1} to be non-positive, we need to have
\begin{align*}
c_{i+1}&\leq c_i -\left(4c_1\alpha^2(1-\beta^k)^2+2\alpha^3\rho_0L^2\frac{(1-\beta^k)^2}{(1-\beta)^2}\right)a_{k,k-i}
\end{align*}
for all $i\geq 1$.

Since $1-\beta^k<1$, it suffices to enforce the following for all $i\geq 1$:
\begin{align}
\label{equ: c i choice}
c_{i+1}&= c_i -\left(4c_1\alpha^2+2\alpha^3\rho_0L^2\frac{1}{(1-\beta)^2}\right)\beta^i(i+\frac{\beta}{1-\beta})L^2.
\end{align}
And in order for $c_i>0$ for all $i\geq 1$, we can determine $c_1$ by
\begin{align*}
c_1 = \left(4c_1\alpha^2+2\alpha^3\rho_0L^2\frac{1}{(1-\beta)^2}\right)\sum_{i=1}^{\infty}\beta^i(i+\frac{\beta}{1-\beta})L^2.
\end{align*}
Since 
\[
\sum_{i=1}^j i \beta^i = \frac{1}{1-\beta}\left(\frac{\beta(1-\beta^j)}{1-\beta}-j\beta^{j+1}\right),
\]
we have $\sum_{i=1}^{\infty} i\beta^i = \frac{\beta}{(1-\beta)^2}$ and
\begin{align*}
c_1=\left(4c_1\alpha^2+2\alpha^3\rho_0L^2\frac{1}{(1-\beta)^2}\right)\frac{\beta+\beta^2}{(1-\beta)^2}L^2.
\end{align*}
This stipulates that
\begin{align}
\label{equ: c 1 choice}
c_1 = \frac{2\alpha^3\rho_0L^4\frac{\beta+\beta^2}{(1-\beta)^4}}{1-4\alpha^2\frac{\beta+\beta^2}{(1-\beta)^2}L^2}.
\end{align}
Notice that $\alpha\leq\frac{1-\beta}{4L\sqrt{\beta+\beta^2}}$ ensures $c_1>0$. 

With the choices of $c_i$ in \eqref{equ: c i choice} and \eqref{equ: c 1 choice}, the sum of the last three terms of \eqref{equ: L^k intermediate 1} is non-positive. Therefore,
\begin{align}
\label{equ: L^k intermediate 2}
\begin{split}
\E[L^{k+1}-L^k]
&\leq \bigg(-\alpha+\alpha\frac{1}{2\rho_0}+2\alpha^3\rho_0L^2(\frac{\beta}{1-\beta})^2+\frac{L\alpha^2}{2}+4c_1\alpha^2\bigg)\E[\|g^k\|^2]\\
&\quad +\bigg(\alpha^3\rho_0L^2(\frac{\beta}{1-\beta})^2\frac{1-\beta}{1+\beta}\sigma^2+\frac{1}{2}L\alpha^2\sigma^2+2c_1\frac{1-\beta}{1+\beta}\alpha^2\sigma^2\bigg).
\end{split}
\end{align}
Finally, taking 
\begin{align}
\label{equ: rho 0 choice}
    \rho_0 = \frac{1-\beta}{2L\alpha}
\end{align}
in \eqref{equ: c 1 choice}, \eqref{equ: c i choice}, and \eqref{equ: L^k intermediate 2} gives
\begin{align*}
c_1 &= \frac{\frac{\beta+\beta^2}{(1-\beta)^3}L^3\alpha^2}{1-4\alpha^2\frac{\beta+\beta^2}{(1-\beta)^2}L^2},
\end{align*}
\begin{align*}
c_{i+1}&=c_i - \left(4c_1\alpha^2+\frac{L\alpha^2}{(1-\beta)}\right)\beta^i(i+\frac{\beta}{1-\beta})L^2,
\end{align*}
\begin{align*}
\E[L^{k+1}-L^k]
&\leq \left(-\alpha+\frac{3-\beta+2\beta^2}{2(1-\beta)}L\alpha^2+4c_1\alpha^2\right)\E[\|g^k\|^2]\\
&\quad +\left(\frac{\beta^2}{2(1+\beta)}L\alpha^2\sigma^2+\frac{1}{2}L\alpha^2\sigma^2+2c_1\frac{1-\beta}{1+\beta}\alpha^2\sigma^2\right).
\end{align*}

\subsection{Proof of Theorem \ref{thm: nonconvex constant}}

From \eqref{equ: L^k intermediate 2} we know that
\begin{align}
\label{equ: L^k intermediate 3}
\begin{split}
\E[L^{k+1}-L^k]\leq -R_1\E[\|g^k\|^2] +R_2,
\end{split}
\end{align}
where
\begin{align}
    R_1 &= \alpha-\alpha\frac{1}{2\rho_0}-2\alpha^3\rho_0L^2(\frac{\beta}{1-\beta})^2-\frac{L\alpha^2}{2}-4c_1\alpha^2\label{equ: R_1}\\
    R_2 &= \alpha^3 \rho_0 L^2(\frac{\beta}{1-\beta})^2\frac{1-\beta}{1+\beta}\sigma^2+\frac{1}{2}L\alpha^2\sigma^2+2c_1\frac{1-\beta}{1+\beta}\alpha^2\sigma^2,\label{equ: R_2}
\end{align}
and $\rho_0 = \frac{1-\beta}{2L\alpha}$.

This immediately tells us that
\begin{align*}
L^1\geq\E[L^1-L^{k+1}]\geq R_1\sum_{i=1}^k\E[\|g^i\|^2]-k R_2,
\end{align*}
and therefore
\begin{align}
\label{equ: 2}
\frac{1}{k}\sum_{i=1}^k \E[\|g^k\|^2]\leq \frac{L^1}{k R_1}+\frac{R_2}{R_1}.
\end{align}

In the rest the proof, we will bound $R_1$ and $R_2$ appropriately. 

First, let us show that $R_1\geq \frac{\alpha}{2}$ when $\rho_0 = \frac{1-\beta}{2L\alpha}$ as in \eqref{equ: rho 0 choice} and $\alpha\leq \min\{\frac{1-\beta}{L(4-\beta+\beta^2)}, \frac{1-\beta}{2L\sqrt{\beta+\beta^2}}\}$.

From \eqref{equ: c 1 choice} we know that
\begin{align*}
c_1 = \frac{2\alpha^3\rho_0L^4\frac{\beta+\beta^2}{(1-\beta)^4}}{1-4\alpha^2\frac{\beta+\beta^2}{(1-\beta)^2}L^2}.
\end{align*}
Since $\alpha\leq \frac{1-\beta}{2\sqrt{2} L\sqrt{\beta+\beta^2}}$, we have
\begin{align}
\label{equ: denominator}
4\alpha^2\frac{\beta+\beta^2}{(1-\beta)^2}L^2\leq \frac{1}{2}
\end{align}
and
\begin{align}
\label{equ: c 1 bound}
c_1\leq 4\alpha^3\rho_0L^4\frac{\beta+\beta^2}{(1-\beta)^4}\leq \frac{1}{4}\alpha\rho_0\frac{L^2}{(1-\beta)^2}.
\end{align}
Therefore, in order to ensure $R_1\geq \frac{\alpha}{2}$ where $R_1$ is defined in \eqref{equ: R_1}, it suffices to have
\begin{align}
\label{equ: 1}
\begin{split}
\alpha\frac{1}{2\rho_0}&+2\alpha \rho_0L^2(\frac{\beta}{1-\beta})^2\alpha^2+\frac{L\alpha^2}{2}+\alpha\rho_0L^2\frac{1}{(1-\beta)^2}\alpha^2\leq \frac{\alpha}{2}.
\end{split}
\end{align}
Applying $\rho_0=\frac{1-\beta}{2L\alpha}$ yields
\begin{align*}
&\alpha\frac{1}{2\rho_0}+2\alpha \rho_0L^2(\frac{\beta}{1-\beta})^2\alpha^2+\frac{L\alpha^2}{2}+\alpha\rho_0L^2\frac{1}{(1-\beta)^2}\alpha^2\\
&= \frac{L\alpha^2}{1-\beta}+\alpha^2L\frac{\beta^2}{1-\beta}+\frac{1}{2}\alpha^2L\frac{1}{1-\beta}+\frac{L\alpha^2}{2}\\
& = L\alpha^2 \left(\frac{1}{1-\beta}+\frac{\beta^2}{1-\beta}+\frac{1}{2}\frac{1}{1-\beta}+\frac{1}{2}\right)\\
& = L\alpha^2\frac{4-\beta+2\beta^2}{2(1-\beta)}\\
&\leq \frac{\alpha}{2},
\end{align*}
where we have applied $\alpha\leq \frac{1-\beta}{L(4-\beta+2\beta^2)}$ in the last step.

Therefore, \eqref{equ: 1} is true and 
\begin{align}
\label{equ: R 1 bound}
R_1\geq \frac{\alpha}{2}.
\end{align}

Now let us turn to $R_2$. By \eqref{equ: c 1 bound} we know that
\begin{align*}
R_2&= \alpha{\rho_0}L^2(\frac{\beta}{1-\beta})^2\alpha^2\frac{1-\beta}{1+\beta}\sigma^2+\frac{1}{2}L\alpha^2\sigma^2+2c_1\frac{1-\beta}{1+\beta}\alpha^2\sigma^2\\
&\leq \alpha{\rho_0}L^2(\frac{\beta}{1-\beta})^2\alpha^2\frac{1-\beta}{1+\beta}\sigma^2+\frac{1}{2}L\alpha^2\sigma^2+8\alpha^3\rho_0L^4\frac{\beta+\beta^2}{(1-\beta)^4}\frac{1-\beta}{1+\beta}\alpha^2\sigma^2.
\end{align*}
Since $\rho_0 = \frac{1-\beta}{2L\alpha}$, we have
\begin{align*}
\begin{split}
R_2&\leq \frac{\beta^2}{2(1+\beta)}L\alpha^2\sigma^2+\frac{1}{2}L\alpha^2\sigma^2+\frac{4\beta}{(1-\beta)^2}L^3\alpha^4\sigma^2.
\end{split}
\end{align*}
By applying $\alpha\leq \min\{\frac{1-\beta}{L(4-\beta+\beta^2)}, \frac{1-\beta}{2\sqrt{2}L\sqrt{\beta+\beta^2}}\}\leq \frac{1-\beta}{3.75 L}<\frac{1-\beta}{4 L}$, we further have
\begin{align}
\label{equ: R 2 bound}
\begin{split}
R_2&\leq \frac{\beta^2}{2(1+\beta)}L\alpha^2\sigma^2+\frac{1}{2}L\alpha^2\sigma^2+\frac{\beta}{4} L\alpha^2\sigma^2.
\end{split}
\end{align}
By putting \eqref{equ: R 1 bound} and \eqref{equ: R 2 bound} into \eqref{equ: 2}, we finally obtain
\begin{align*}
\begin{split}
\frac{1}{k}\sum_{i=1}^k\E[\|g^i\|^2]
&\leq \frac{2\left(f(x^1)-f^*\right)}{k\alpha}+\left(\frac{\beta+3\beta^2}{2(1+\beta)}+1\right)L\alpha\sigma^2\\
&=\mathcal{O}\left(\frac{f(x^1)-f^*}{k\alpha}\right)+\mathcal{O}\left(L\alpha\sigma^2\right).
\end{split}  
\end{align*}

\subsection{Proof of Proposition \ref{prop: L^k scvx}}
\label{app: l^k scvx}
In order to prove Proposition \ref{prop: L^k scvx}, we will set
\begin{align*}
c_1 = \left(\frac{\sqrt{\beta}}{(1-\sqrt{\beta})^2}+\frac{\sqrt{\beta}}{1-\sqrt{\beta}}\frac{\beta}{1-\beta}\right)\left(\frac{2L^3\alpha^2}{1-\beta} + \frac{18L^2\mu\alpha^2}{(1-\beta)(1+\frac{8\mu}{L})}\right),
\end{align*}
\begin{align*}
c_{i+1}-c_i+ A_3 2 L^2\beta^{k-i}\left(k-i+\frac{\beta}{1-\beta}\right) = A_1 c_i, \quad\quad \forall i\geq 1,
\end{align*}
Take $\rho_0 = \frac{1-\beta}{2L\alpha}$ in \eqref{equ: L^k intermediate 1}, we have
\begin{align}
\label{equ: L^k scvx 1}
\begin{split}
\E[L^{k+1}-L^k]
&\leq \left(-\alpha+\frac{3-\beta+2\beta^2}{2(1-\beta)}L\alpha^2+4c_1\alpha^2\right)\E[\|g^k\|^2]\\
&\quad +\left(\frac{\beta^2}{2(1+\beta)}L\alpha^2\sigma^2+\frac{1}{2}L\alpha^2\sigma^2+2c_1\frac{1-\beta}{1+\beta}\alpha^2\sigma^2\right)\\
&\quad + \sum_{i=1}^{k-1}(c_{i+1}-c_i)\E[\|x^{k+1-i}-x^{k-i}\|^2]\\
&\quad +\left(4c_1\alpha^2 + \frac{L\alpha^2}{(1-\beta)}\right) (1-\beta^k)^2\E[\|\frac{1}{1-\beta^{k}}(1-\beta)\sum_{i=1}^k \beta^{k-i}g^i-g^k\|^2].
\end{split}
\end{align}

Let us first derive a lower bound of the first term on the right hand side of \eqref{equ: L^k scvx 1}.

From the strong convexity of $f$ we have 
\begin{align}
\label{equ: scvx lemma}
\E[\|g^k\|^2]=\E[\|\nabla f(x^k)\|^2]\geq 2\mu \E[f(x^k)-f^{\star}],
\end{align}
where $f^{\star} = \min_{x\in\Rd}f(x)$. On the other hand, for $\E[f(x^k)]$ we have
\begin{align*}
\E[f(z^k)]&\leq \E[f(x^k)] +\E[\langle g^k, z^k-x^k\rangle] + \frac{L}{2}\E[\|z^k-x^k\|^2]\\
&= \E[f(x^k)] +\E[\langle g^k - \frac{1}{1-\beta^k}(1-\beta)\sum_{i=1}^k \beta^{k-i}g^i+ \frac{1}{1-\beta^k}(1-\beta)\sum_{i=1}^k \beta^{k-i}g^i, -\frac{\alpha \beta}{1-\beta}m^{k-1}\rangle] \\
&\,\,\, + \frac{L}{2}\E[\|\frac{\alpha \beta}{1-\beta}m^{k-1}\|^2]\\
&\leq \E[f(x^k)] + \alpha\frac{\rho}{2}\E[\|g^k-\frac{1}{1-\beta^k}(1-\beta)\sum_{i=1}^k \beta^{k-i}g^i\|^2]+\frac{\alpha}{2\rho}\E[\|\frac{\beta}{1-\beta}m^{k-1}\|^2]\\
&\,\,\,+\E[\langle \frac{1}{1-\beta^k}(1-\beta)\sum_{i=1}^k \beta^{k-i}g^i, -\frac{\alpha\beta}{1-\beta}m^{k-1}\rangle]+\frac{L}{2}\E[\|\frac{\alpha \beta}{1-\beta}m^{k-1}\|^2]\\
&\leq \E[f(x^k)] + \alpha\frac{\rho}{2}\E[\|g^k-\frac{1}{1-\beta^k}(1-\beta)\sum_{i=1}^k \beta^{k-i}g^i\|^2]\\
&\,\,\, +\left(\frac{\alpha}{2\rho}(\frac{\beta}{1-\beta})^2+\frac{L\alpha^2}{2}(\frac{\beta}{1-\beta})^2\right)\E[\|m^{k-1}\|^2]\\
&\,\,\,+ \alpha\frac{\beta}{1-\beta}\left(\frac{\rho_1}{2}\E[\|\frac{1}{1-\beta^k}(1-\beta)\sum_{i=1}^k \beta^{k-i}g^i\|^2]+\frac{1}{2\rho_1}\E[\|m^{k-1}\|^2]\right)\\
&= \E[f(x^k)] + \alpha\frac{\rho}{2}\E[\|g^k-\frac{1}{1-\beta^k}(1-\beta)\sum_{i=1}^k \beta^{k-i}g^i\|^2]\\
&\,\,\, +\left(\frac{\alpha}{2\rho}(\frac{\beta}{1-\beta})^2+\frac{L\alpha^2}{2}(\frac{\beta}{1-\beta})^2+\alpha\frac{\beta}{1-\beta}\frac{1}{2\rho_1}\right)\E[\|m^{k-1}\|^2]\\
&\,\,\,+ \alpha\frac{\beta}{1-\beta}\frac{\rho_1}{2}\E[\|\frac{1}{1-\beta^k}(1-\beta)\sum_{i=1}^k \beta^{k-i}g^i\|^2],
\end{align*}
where $\rho$, $\rho_1>0$ are to be determined later. 

Combining this with \eqref{equ: scvx lemma} gives
\begin{align}
\label{equ: aaa}
\begin{split}
\E[\|g^k\|^2]\geq 2\mu\bigg(&\E[f(z^k)] - f^{\star} - \alpha\frac{\rho}{2}\E[\|g^k-\frac{1}{1-\beta^k}(1-\beta)\sum_{i=1}^k \beta^{k-i}g^i\|^2]\\
&\,\,\, -\left(\frac{\alpha}{2\rho}(\frac{\beta}{1-\beta})^2+\frac{L\alpha^2}{2}(\frac{\beta}{1-\beta})^2+\alpha\frac{\beta}{1-\beta}\frac{1}{2\rho_1}\right)\E[\|m^{k-1}\|^2]\\
&\,\,\,- \alpha\frac{\beta}{1-\beta}\frac{\rho_1}{2}\E[\|\frac{1}{1-\beta^k}(1-\beta)\sum_{i=1}^k \beta^{k-i}g^i\|^2],\bigg).
\end{split}
\end{align}
On the other hand, we have from \eqref{equ: useful inequalities 1} that 
\begin{align*}
\E[\|m^{k-1}\|^2]&\leq  2\frac{1-\beta}{1+\beta}\sigma^2 + 2(1-\beta^{k-1})^2\left(2\E[\|g^{k}\|^2]+2\E[\|\frac{1}{1-\beta^{k-1}}(1-\beta)\sum_{i=1}^{k-1} \beta^{k-1-i} g^i-g^k\|^2]\right)\\
& = 2\frac{1-\beta}{1+\beta}\sigma^2 \\
&\,\,\, + 2(1-\beta^{k-1})^2\left(2\E[\|g^{k}\|^2]+2\frac{1}{\beta^2}(\frac{1-\beta^k}{1-\beta^{k-1}})^2\E[\|\frac{1}{1-\beta^{k}}(1-\beta)\sum_{i=1}^{k} \beta^{k-i} g^i-g^k\|^2]\right),
\end{align*}
and that 
\begin{align*}
\E[\|\frac{1}{1-\beta^k}(1-\beta)\sum_{i=1}^k \beta^{k-i}g^i\|^2]\leq 2\E[\|g^k\|^2]+2\E[\|\frac{1}{1-\beta^{k}}(1-\beta)\sum_{i=1}^{k} \beta^{k-i} g^i-g^k\|^2].
\end{align*}
Putting these two inequalities into \eqref{equ: aaa} and rearranging gives 
\begin{align*}
&\bigg[1+2\mu\bigg(\left(\frac{\alpha}{2\rho}(\frac{\beta}{1-\beta})^2+\frac{L\alpha^2}{2}(\frac{\beta}{1-\beta})^2+\alpha\frac{\beta}{1-\beta}\frac{1}{2\rho_1}\right)4(1-\beta^{k-1})^2\\
&\qquad\qquad\qquad\qquad +\alpha\frac{\beta}{1-\beta}\rho_1\bigg)\bigg]\E[\|g^k\|^2]\\
&\geq 2\mu\bigg[\E[f(z^k)] - f^{\star} - \alpha\frac{\rho}{2}\E[\|g^k-\frac{1}{1-\beta^k}(1-\beta)\sum_{i=1}^k \beta^{k-i}g^i\|^2]\\
&\,\,\, -\left(\frac{\alpha}{2\rho}(\frac{\beta}{1-\beta})^2+\frac{L\alpha^2}{2}(\frac{\beta}{1-\beta})^2+\alpha\frac{\beta}{1-\beta}\frac{1}{2\rho_1}\right)\\
&\,\,\,\,\,\,\,\,\,\times\left(2\frac{1-\beta}{1+\beta}\sigma^2 + (1-\beta^{k-1})^2 4\frac{1}{\beta^2}(\frac{1-\beta^k}{1-\beta^{k-1}})^2\E[\|\frac{1}{1-\beta^{k}}(1-\beta)\sum_{i=1}^{k} \beta^{k-i} g^i-g^k\|^2]\right)\\
&\,\,\,- \alpha\frac{\beta}{1-\beta}\frac{\rho_1}{2}2\E[\|\frac{1}{1-\beta^{k}}(1-\beta)\sum_{i=1}^{k} \beta^{k-i} g^i-g^k\|^2\bigg]\\
&=2\mu\bigg[\E[f(z^k)] - f^{\star}\\
&\,\,\, -\left(\frac{\alpha}{2\rho}(\frac{\beta}{1-\beta})^2+\frac{L\alpha^2}{2}(\frac{\beta}{1-\beta})^2+\alpha\frac{\beta}{1-\beta}\frac{1}{2\rho_1}\right)2\frac{1-\beta}{1+\beta}\sigma^2\\
&\,\,\, -\left(\alpha\frac{\rho}{2}+\left(\frac{\alpha}{2\rho}(\frac{\beta}{1-\beta})^2+\frac{L\alpha^2}{2}(\frac{\beta}{1-\beta})^2+\alpha\frac{\beta}{1-\beta}\frac{1}{2\rho_1}\right)4(1-\beta^k)^2\frac{1}{\beta^2}+\alpha\frac{\beta}{1-\beta}\rho_1\right)\\
&\,\,\,\,\,\,\,\,\,\times \E[\|\frac{1}{1-\beta^{k}}(1-\beta)\sum_{i=1}^{k} \beta^{k-i} g^i-g^k\|^2\bigg].
\end{align*}
Taking $\rho= \frac{1}{1-\beta}$ and $\rho_1 = \frac{1}{\beta}$ gives
\begin{align}
\label{equ: bbb}
\begin{split}
&\left[1+2\mu\left(\left({\alpha}\frac{\beta^2}{1-\beta}+\frac{L\alpha^2}{2}(\frac{\beta}{1-\beta})^2\right)4(1-\beta^{k-1})^2+\alpha\frac{1}{1-\beta}\right)\right]\E[\|g^k\|^2]\\
&\geq 2\mu\bigg[\E[f(z^k)] - f^{\star}\\
&\,\,\, -\left({\alpha}\frac{\beta^2}{1-\beta}+\frac{L\alpha^2}{2}(\frac{\beta}{1-\beta})^2\right)2\frac{1-\beta}{1+\beta}\sigma^2\\
&\,\,\, -\left(\alpha\frac{1}{2(1-\beta)}+\left({\alpha}\frac{\beta^2}{1-\beta} + \frac{L\alpha^2}{2}(\frac{\beta}{1-\beta})^2\right)4(1-\beta^k)^2\frac{1}{\beta^2}+\alpha\frac{1}{1-\beta}\right)\\
&\,\,\,\,\,\,\,\,\,\times \E[\|\frac{1}{1-\beta^{k}}(1-\beta)\sum_{i=1}^k \beta^{k-i}g^i-g^k\|^2\bigg].
\end{split}
\end{align}
{Since }
\begin{align}
\label{equ: scvx step size 1}
\alpha \leq \frac{1-\beta}{5 L},
\end{align} 
\eqref{equ: bbb} gives
\begin{align}
\label{equ: ccc}
\begin{split}
&\left(1+8\frac{\mu}{L}\right)\E[\|g^k\|^2]\\
&\geq 2\mu\bigg[\E[f(z^k)] - f^{\star}\\
&\,\,\, -\left({\alpha}\frac{\beta^2}{1-\beta}+\frac{L\alpha^2}{2}(\frac{\beta}{1-\beta})^2\right)2\frac{1-\beta}{1+\beta}\sigma^2\\
&\,\,\, -\left(\alpha\frac{1}{2(1-\beta)}+\left({\alpha}\frac{\beta^2}{1-\beta} + \frac{L\alpha^2}{2}(\frac{\beta}{1-\beta})^2\right)4(1-\beta^k)^2\frac{1}{\beta^2}+\alpha\frac{1}{1-\beta}\right)\\
&\,\,\,\,\,\,\,\,\,\times \E[\|\frac{1}{1-\beta^{k}}(1-\beta)\sum_{i=1}^k \beta^{k-i}g^i-g^k\|^2\bigg].
\end{split}
\end{align}
Since $\alpha\leq \frac{1-\beta}{5 L}$, we have that
\begin{align}
\label{equ: c_1 scvx estimate}
\begin{split}
c_1 &= \left(\frac{\sqrt{\beta}}{(1-\sqrt{\beta})^2}+\frac{\sqrt{\beta}}{1-\sqrt{\beta}}\frac{\beta}{1-\beta}\right)\left(\frac{2L^3\alpha^2}{1-\beta} + \frac{18L^2\mu\alpha^2}{(1-\beta)(1+\frac{8\mu}{L})}\right)\\
&\leq  \left(\frac{4\sqrt{\beta}}{(1-{\beta})^2}+\frac{2\sqrt{\beta}}{1-{\beta}}\frac{\beta}{1-\beta}\right)\left(\frac{2L(1-\beta)}{25} + \frac{18\mu(1-\beta)}{25(1+\frac{8\mu}{L})}\right)\\
&\leq \frac{6\sqrt{\beta}}{25(1-\beta)}\left(2L+\frac{18\mu}{1+\frac{8\mu}{L}}\right)\\
&\leq \frac{6\sqrt{\beta}}{25(1-\beta)}\left(2L+18\mu\right)
\end{split}
\end{align}
Therefore, by $\alpha\leq \frac{1-\beta}{L\left(3-\beta+2\beta^2+\frac{48\sqrt{\beta}}{25}\frac{2L+18\mu}{L}\right)}$ we have
\begin{align}
\label{equ: scvx step size 2}
\begin{split}
&-\alpha+\frac{3-\beta+2\beta^2}{2(1-\beta)}L\alpha^2+4c_1\alpha^2\\
& = -\frac{\alpha}{2}-\frac{\alpha}{2}+ \frac{3-\beta+2\beta^2}{2(1-\beta)}L\alpha^2+ \frac{24\sqrt{\beta}}{25(1-\beta)}\left(2L+18\mu\right)\alpha^2\\
&\leq -\frac{\alpha}{2}.
\end{split}
\end{align}
Combine \eqref{equ: scvx step size 2} with \eqref{equ: L^k scvx 1}, we have
\begin{align}
\label{equ: L^k scvx 2}
\begin{split}
\E[L^{k+1}-L^k]
&\leq -\frac{\alpha}{2}\E[\|g^k\|^2] +\left(\frac{\beta^2}{2(1+\beta)}L\alpha^2\sigma^2+\frac{1}{2}L\alpha^2\sigma^2+2c_1\frac{1-\beta}{1+\beta}\alpha^2\sigma^2\right)\\
&\quad + \sum_{i=1}^{k-1}(c_{i+1}-c_i)\E[\|x^{k+1-i}-x^{k-i}\|^2]\\
&\quad + \left(4c_1\alpha^2 + \frac{L\alpha^2}{(1-\beta)}\right) (1-\beta^k)^2\E[\|\frac{1}{1-\beta^{k}}(1-\beta)\sum_{i=1}^k \beta^{k-i}g^i-g^k\|^2].
\end{split}
\end{align}
By combining \eqref{equ: L^k scvx 2} with \eqref{equ: ccc}, we further obtain
\begin{align}
\label{equ: L^k scvx 3}
\begin{split}
\E[L^{k+1}-L^k]&\leq B_1 \E[f(z^k) - f^{\star}] + B_2\\
&\quad + B_3 \E[\|\frac{1}{1-\beta^{k}}(1-\beta)\sum_{i=1}^k \beta^{k-i}g^i-g^k\|^2]+ \sum_{i=1}^{k-1}(c_{i+1}-c_i)\E[\|x^{k+1-i}-x^{k-i}\|^2],
\end{split}
\end{align}
where
\begin{align}
\label{equ: scvx A}
\begin{split}
B_1 &= -\frac{\alpha}{2}\frac{2\mu}{1+\frac{8\mu}{L}},\\
B_2 &= \frac{\beta^2}{2(1+\beta)}L\alpha^2\sigma^2+\frac{1}{2}L\alpha^2\sigma^2+2c_1\frac{1-\beta}{1+\beta}\alpha^2\sigma^2 \\
&\,\,\, + \frac{\alpha}{2}\frac{2\mu\left({\alpha}\frac{\beta^2}{1-\beta}+\frac{L\alpha^2}{2}(\frac{\beta}{1-\beta})^2\right)2\frac{1-\beta}{1+\beta}\sigma^2}{1+\frac{8\mu}{L}},\\
B_3 &= 4c_1\alpha^2 + \frac{L\alpha^2}{(1-\beta)} \\
&\,\,\, + \frac{\alpha}{2}\frac{2\mu \left(\alpha\frac{1}{2(1-\beta)}+\left({\alpha}\frac{\beta^2}{1-\beta} +\frac{L\alpha^2}{2}(\frac{\beta}{1-\beta})^2\right)4\frac{1}{\beta^2}+\alpha\frac{1}{1-\beta}\right)}{1+\frac{8\mu}{L}}.
\end{split}
\end{align}
From Lemma \ref{lem: difference} we know that
\begin{align*}
\E\left[\|\frac{1}{1-\beta^k}(1-\beta)\sum_{i=1}^k \beta^{k-i}g^i-g^k\|^2\right]\leq \sum_{i=1}^{k-1}a_{k,i}\E[\|x^{i+1}-x^i\|^2],
\end{align*}
where 
\begin{align}
\label{equ: scvx a k i}
a_{k, i}=\frac{L^2\beta^{k-i}}{1-\beta^k}\left(k-i+\frac{\beta}{1-\beta}\right).
\end{align}
Putting this into \eqref{equ: L^k scvx 3} yields
\begin{align}
\label{equ: L^k scvx 4}
\begin{split}
\E[L^{k+1}-L^k]&\leq B_1 \E[f(z^k) - f^{\star}] + B_2\\
&\quad + \sum_{i=1}^{k-1}(c_{i+1}-c_i+ B_3 a_{k, k-i})\E[\|x^{k+1-i}-x^{k-i}\|^2.
\end{split}
\end{align}
In the rest of the proof, we will show that if the constants $c_i$ are chosen such that 
\begin{align}
\label{equ: scvx choice of c 1}
c_1 = \left(\frac{\sqrt{\beta}}{(1-\sqrt{\beta})^2}+\frac{\sqrt{\beta}}{1-\sqrt{\beta}}\frac{\beta}{1-\beta}\right)\left(\frac{4 L^3\alpha^2}{1-\beta} + \frac{30 L^2\mu\alpha^2}{(1-\beta)(1+\frac{8\mu}{L})}\right),
\end{align}
and
\begin{align}
\label{equ: scvx choice of c i}
c_{i+1}-c_i+ B_3 2 L^2\beta^{k-i}\left(k-i+\frac{\beta}{1-\beta}\right) = B_1 c_i, \quad\quad \forall i\geq 1.
\end{align}
Then, we have $c_i>0$ for all $i\geq 1$ and 
\begin{align}
\label{equ: scvx c_i}
c_{i+1}-c_i+ B_3 a_{k, k-i} \leq B_1 c_i, \quad\quad \forall i \geq 1.
\end{align}
And therefore, we will have the desired result:
\begin{align*}
\E[L^{k+1}-L^k]&\leq B_1 \E[f(z^k) - f^{\star}] + B_2 + B_1\sum_{i=1}^{k-1} c_i\E[\|x^{k+1-i}-x^{k-i}\|^2\\
& = -\frac{\alpha\mu}{1+\frac{8\mu}{L}}\E[L^k] + \frac{\beta^2}{2(1+\beta)}L\alpha^2\sigma^2+\frac{1}{2}L\alpha^2\sigma^2+2c_1\frac{1-\beta}{1+\beta}\alpha^2\sigma^2 \\
&\,\,\, + \frac{{\beta^2}+\frac{L\alpha}{2}\frac{\beta^2}{1-\beta}}{1+\frac{8\mu}{L}}\frac{2}{1+\beta}\mu\alpha^2\sigma^2.
\end{align*}
First of all. {by $k\geq \frac{\log 0.5}{\log \beta}$,} we know that $\beta^k\leq \frac{1}{2}$, and \eqref{equ: scvx a k i} gives
\[
a_{k, k-i}\leq 2 L^2\beta^{i}\left(i+\frac{\beta}{1-\beta}\right).
\]
Therefore, in order for \eqref{equ: scvx c_i} to hold, it suffices to set 
\[
c_{i+1}-c_i+ B_3 2 L^2\beta^{k-i}\left(k-i+\frac{\beta}{1-\beta}\right) = B_1 c_i \quad\quad \forall i\geq 1.
\]
This is exactly \eqref{equ: scvx choice of c i}. 

On the other hand, \eqref{equ: scvx choice of c i} is also equivalent to 
\[
\frac{c_{i+1}}{(1+B_1)^{i+1}}-\frac{c_i}{(1+B_1)^i}=  -\frac{2 L^2 B_3}{(1+B_1)^{i+1}}\beta^{i}\left(i+\frac{\beta}{1-\beta}\right), \quad\quad \forall i \geq 1.
\]
Therefore, in order to have $c_i>0$ for all $i\geq 1$, we can set
\begin{align}
\label{equ: scvx c 1 aaa}
c_1 & \geq  2L^2B_3\sum_{i=1}^{\infty} \left(\frac{\beta}{1+B_1}\right)^{i}\left(i+\frac{\beta}{1-\beta}\right).
\end{align}
{Since $\beta\leq \sqrt{\beta}\leq 1+B_1 = 1-\alpha\mu \frac{1}{1+\frac{8\mu}{L}}$} and 
\[
\sum_{i=1}^j i q^i = \frac{1}{1-q}\left(\frac{q(1-q^j)}{1-q}-j q^{j+1}\right),
\] 
for any $q\in(0,1)$, \eqref{equ: scvx c 1 aaa} is equivalent to
\begin{align}
\label{equ: scvx c 1 bbb}
c_1 \geq 2L^2B_3\left(\frac{\frac{\beta}{1+B_1}}{(1-\frac{\beta}{1+B_1})^2} + \frac{\frac{\beta}{1+B_1}}{1-\frac{\beta}{1+B_1}}\frac{\beta}{1-\beta}\right).
\end{align}
Recall from \eqref{equ: scvx A} that 
\begin{align*}
B_3 &= 4c_1\alpha^2 + \frac{L\alpha^2}{(1-\beta)} \\
&\,\,\, + \frac{\alpha}{2}\frac{2\mu \left(\alpha\frac{1}{2(1-\beta)}+\left({\alpha}\frac{\beta^2}{1-\beta} +\frac{L\alpha^2}{2}(\frac{\beta}{1-\beta})^2\right)4\frac{1}{\beta^2}+\alpha\frac{1}{1-\beta}\right)}{1+\frac{8\mu}{L}}\\
& = \left(4c_1\alpha^2 + \frac{L\alpha^2}{(1-\beta)}\right) + \frac{\alpha}{2}\frac{2\mu \left(\alpha\frac{11}{2(1-\beta)}+{2L\alpha^2}(\frac{1}{1-\beta})^2\right)}{1+\frac{8\mu}{L}}.
\end{align*}
Since $\alpha\leq \frac{1-\beta}{L}$, we further have
\[
B_3\leq \left(4c_1\alpha^2 + \frac{L\alpha^2}{(1-\beta)}\right) + \frac{\alpha}{2}\frac{2\mu \left(\alpha\frac{15}{2(1-\beta)}\right)}{1+\frac{8\mu}{L}}.
\]
Since $B_1 = -\frac{\alpha \mu}{1+\frac{8\mu}{L}}$ and $\alpha\leq \frac{1-\beta}{5L}$, it can be verified that $\frac{\beta}{1+B_1}\leq \sqrt{\beta}$ for all $\beta\in[0,1)$ and $\mu\leq L$. Therefore,
\[
\frac{\frac{\beta}{1+B_1}}{(1-\frac{\beta}{1+B_1})^2} + \frac{\frac{\beta}{1+B_1}}{1-\frac{\beta}{1+B_1}}\frac{\beta}{1-\beta}\leq \frac{\sqrt{\beta}}{(1-\sqrt{\beta})^2}+\frac{\sqrt{\beta}}{1-\sqrt{\beta}}\frac{\beta}{1-\beta}.
\]
As a result, in order to have \eqref{equ: scvx c 1 bbb}, it suffices to set 
\begin{align}
\label{equ: scvx c 1 ccc}
c_1\geq 2L^2\left(\frac{\sqrt{\beta}}{(1-\sqrt{\beta})^2}+\frac{\sqrt{\beta}}{1-\sqrt{\beta}}\frac{\beta}{1-\beta}\right) \left(4c_1\alpha^2 + \frac{L\alpha^2}{(1-\beta)} + \frac{\alpha}{2}\frac{2\mu \left(\alpha\frac{15}{2(1-\beta)}\right)}{1+\frac{8\mu}{L}}\right),
\end{align}
Since $\alpha\leq \frac{1-\beta}{5L}$, we have
\[
1 - 8\alpha^2L^2\left(\frac{\sqrt{\beta}}{(1-\sqrt{\beta})^2}+\frac{\sqrt{\beta}}{1-\sqrt{\beta}}\frac{\beta}{1-\beta}\right) \geq \frac{1}{2},
\]
\eqref{equ: scvx c 1 ccc} in turn just requires 
\[
c_1 = \left(\frac{\sqrt{\beta}}{(1-\sqrt{\beta})^2}+\frac{\sqrt{\beta}}{1-\sqrt{\beta}}\frac{\beta}{1-\beta}\right)\left(\frac{4L^3\alpha^2}{1-\beta} + \frac{30 L^2\mu\alpha^2}{(1-\beta)(1+\frac{8\mu}{L})}\right),
\]
which is exactly our choice of $c_1$ as in \eqref{equ: scvx choice of c 1}.

\subsection{Proof of Theorem \ref{thm: scvx constant}}

From Proposition \ref{prop: L^k scvx} we know that for all $k\geq k_0= \lfloor \frac{\log 0.5}{\log \beta}\rfloor$,
\begin{align*}
\E[L^{k+1}-L^k] & \leq -\frac{\alpha\mu}{1+\frac{8\mu}{L}}\E[L^k]+ \frac{1+\beta+\beta^2}{2(1+\beta)}L\alpha^2\sigma^2+\frac{1-\beta}{1+\beta}2c_1\alpha^2\sigma^2 \\
&\,\,\, + \frac{{\beta^2}+\frac{L\alpha}{2}\frac{\beta^2}{1-\beta}}{(1+\frac{8\mu}{L})(1+\beta)}2\mu\alpha^2\sigma^2.
\end{align*}
Rearranging gives 
\begin{align*}
\E[L^{k+1}] & \leq \left(1-\frac{\alpha\mu}{1+\frac{8\mu}{L}}\right)\E[L^k]+ \frac{1+\beta+\beta^2}{2(1+\beta)}L\alpha^2\sigma^2+\frac{1-\beta}{1+\beta}2c_1\alpha^2\sigma^2 \\
&\,\,\,+ \frac{{\beta^2}+\frac{L\alpha}{2}\frac{\beta^2}{1-\beta}}{(1+\frac{8\mu}{L})(1+\beta)}2\mu\alpha^2\sigma^2 \\
&\leq \left(1-\frac{\alpha\mu}{1+\frac{8\mu}{L}}\right)\E[L^k]+ \frac{1+\beta+\beta^2}{2(1+\beta)}L\alpha^2\sigma^2+\frac{1-\beta}{1+\beta}2c_1\alpha^2\sigma^2\\
&\,\,\,+ \frac{{\beta^2}+\frac{L\alpha}{10}{\beta^2}}{1+\frac{8\mu}{L}}\frac{2}{1+\beta}\mu\alpha^2\sigma^2,
\end{align*}
where we have applied $\alpha\leq \frac{1-\beta}{5L}$ in the last step. Therefore,
\begin{align*}
&\E[L^{k+1}] - \frac{1}{\frac{\alpha\mu}{1+\frac{8\mu}{L}}}\left( \frac{1+\beta+\beta^2}{2(1+\beta)}L\alpha^2\sigma^2+\frac{1-\beta}{1+\beta}2c_1\alpha^2\sigma^2  + \frac{{\beta^2}+\frac{L\alpha}{10}{\beta^2}}{1+\frac{8\mu}{L}}\frac{2}{1+\beta}\mu\alpha^2\sigma^2\right)\\
&\leq \left(1-\frac{\alpha\mu}{1+\frac{8\mu}{L}}\right)\\
&\,\quad \times \left(\E[L^k]- \frac{1}{\frac{\alpha\mu}{1+\frac{8\mu}{L}}}\left( \frac{1+\beta+\beta^2}{2(1+\beta)}L\alpha^2\sigma^2+\frac{1-\beta}{1+\beta}2c_1\alpha^2\sigma^2  + \frac{{\beta^2}+\frac{L\alpha}{10}{\beta^2}}{1+\frac{8\mu}{L}}\frac{2}{1+\beta}\mu\alpha^2\sigma^2\right)\right).
\end{align*}
This immediately yields
\begin{align*}
&\E[L^{k}]\\
&\leq \left(1-\frac{\alpha\mu}{1+\frac{8\mu}{L}}\right)^{k-k_0}\\
&\,\,\,\times \left(\E[L^{k_0}]- \frac{1}{\frac{\alpha\mu}{1+\frac{8\mu}{L}}}\left( \frac{1+\beta+\beta^2}{2(1+\beta)}L\alpha^2\sigma^2+\frac{1-\beta}{1+\beta}2c_1\alpha^2\sigma^2  + \frac{{\beta^2}+\frac{L\alpha}{10}{\beta^2}}{1+\frac{8\mu}{L}}\frac{2}{1+\beta}\mu\alpha^2\sigma^2\right)\right)\\
&\,\,\,+ \frac{1}{\frac{\alpha\mu}{1+\frac{8\mu}{L}}}\left( \frac{1+\beta+\beta^2}{2(1+\beta)}L\alpha^2\sigma^2+\frac{1-\beta}{1+\beta}2c_1\alpha^2\sigma^2  + \frac{{\beta^2}+\frac{L\alpha}{10}{\beta^2}}{1+\frac{8\mu}{L}}\frac{2}{1+\beta}\mu\alpha^2\sigma^2\right)\\
&\leq \left(1-\frac{\alpha\mu}{1+\frac{8\mu}{L}}\right)^{k-k_0}\E[L^{k_0}]\\
&\,\,\,+ \left(1+\frac{8\mu}{L}\right)\left( \frac{1+\beta+\beta^2}{4(1+\beta)}\frac{L}{\mu}\alpha\sigma^2+\frac{1-\beta}{1+\beta}\frac{2 c_1}{\mu}\alpha\sigma^2  + \frac{{\beta^2}+\frac{L\alpha}{10}{\beta^2}}{1+\frac{8\mu}{L}}\frac{2}{1+\beta}\alpha\sigma^2\right).
\end{align*}
By $c_i\geq 0$ for all $i\geq 1$ and  \eqref{equ: c_1 scvx estimate}, we conclude that
\begin{align*}
&\E[f(z^k)-f^*]\\
&\leq \left(1-\frac{\alpha\mu}{1+\frac{8\mu}{L}}\right)^{k-k_0}\E[L^{k_0}]\\
&\,\,\,+ \left(1+\frac{8\mu}{L}\right)\left( \frac{1+\beta+\beta^2}{2(1+\beta)}\frac{L}{\mu}\alpha\sigma^2+\frac{1}{1+\beta}\frac{12\sqrt{\beta}}{25}\frac{2L+18\mu}{\mu}\alpha\sigma^2  + \frac{{\beta^2}+\frac{L\alpha}{10}{\beta^2}}{1+\frac{8\mu}{L}}\frac{2}{1+\beta}\alpha\sigma^2\right)\\
& = \mathcal{O}\left((1-\alpha\mu)^{k-k_0} + \frac{L}{\mu}\alpha\sigma^2\right).
\end{align*}

\subsection{Proof of Corollary \ref{coro: to x^k}}
In fact, by \eqref{equ: z^k} we can express $x^k$ as a convex combination of $\{z^i\}_{i=1}^{k}$:
\[
x^k = (1-\beta)\sum_{i=2}^k \beta^{k-i} z^i + \beta^{k-1}z^1.
\]
The desired result follows directly from the convexity of $f$ and Theorem \ref{thm: scvx constant}.

\section{Generalizations of Lemmas \ref{lem: m^k}, \ref{lem: difference}, and \ref{lem: z^k update} for Multistage SGDM}
\label{App: generalizations for multistage}

In order to establish the convergence of Multistage SGDM(Algorithm \ref{alg: Multistage SGDM}), we need to generalize the Lemmas \ref{lem: m^k} and \ref{lem: difference} , which play a key role in the convergence of SGDM in \eqref{equ: momentum}.

\subsection{Generalization of Lemma \ref{lem: m^k} for Multistage SGDM}
\begin{lemma}
\label{lem: m^k for multistage}
Under the assumptions of Theorem \ref{thm: nonconvex Multistage}, the variance of update vector $m^k$ in Algorithm \ref{alg: Multistage SGDM} satisfies 
\begin{align*}
\frac{1}{1-\beta_1}\EE[\|m^k - \sum_{i=1}^k b_{k,i} g^i\|^2] &\leq 2\sigma^2,
\end{align*}
where $b_{k,i}=\big(1-\beta(i)\big)\prod_{j=i+1}^k \beta(j)$.
\end{lemma}
\begin{proof}
To begin with, let us express $m^k$ by the past stochastic gradients:
\begin{align}
\label{equ: m^k multistage}
\begin{split}
m^k&=\beta(k)m^{k-1}+\big(1-\beta(k)\big)\tg^k\\
&=\beta(k)\beta(k-1)m^{k-2}+\beta(k)\big(1-\beta(k-1)\big)\tg^{k-1}\\
&\qquad\qquad\qquad\qquad+ \dots +\big(1-\beta(k)\big)\tg^k\\
&=...\\
&=\prod_{i=1}^k\beta(i) m^0+ \prod_{i=2}^k\beta(i) \big(1-\beta(1)\big)\tg^1\\
&\qquad\qquad\qquad\qquad+ \dots +\big(1-\beta(k)\big)\tg^k\\
&=\sum_{i=1}^k b_{k,i}\tg^i,
\end{split}
\end{align}
where we have applied $m^0=0$ and defined 
\begin{align}
\label{equ: b k i}
b_{k,i}=\big(1-\beta(i)\big)\prod_{j=i+1}^k \beta(j)
\end{align} 
in the last step.

It can be verified that the sum of weights is 
\begin{align}
\label{equ: sum of weights}
\sum_{i=1}^k b_{k,i} = 1-\prod_{i=1}^k \beta(i).
\end{align}

As a result, by applying Assumption \ref{assump: standard assumption} we have
\begin{align*}
\EE[\|m^k - \sum_{i=1}^k b_{k,i} g^i\|^2]=\E[\|\sum_{i=1}^{k}b_{k,i}(\tg^i-g^i)\|^2]\leq\sum_{i=1}^k b^2_{k,i}\sigma^2.
\end{align*}
Note that by setting $k=T_1+ \dots +T_{n_k}+r_k$, we have
\begin{align*}
b_{k,i}=
\begin{cases}
\beta_{n_k+1}^{r_k}\beta_{n_k}^{T_{n_k}}...\beta_2^{T_2}(1-\beta_1)\beta_1^{T_1-i}, 1\leq i \leq T_1,\\
\beta_{n_k+1}^{r_k}\beta_{n_k}^{T_{n_k}}...\beta_3^{T_3}(1-\beta_2)\beta_1^{T_1+T_2-i}, T_1+1\leq i \leq T_1+T_2,\\
.....\\
(1-\beta_{n_k+1})\beta_1^{T_1+ \dots +T_{n_k}+r_k-i}, \sum_{l=1}^{n_k}T_l+1\leq i \leq \sum_{l=1}^{n_k}T_l+r_k.
\end{cases}
\end{align*}
Therefore,
\begin{align*}
\EE[\|m^k - \sum_{i=1}^k b_{k,i} g^i\|^2]
&\leq (\beta^{r_k}_{n_k+1}\beta^{T_{n_k}}_{n_k}...\beta^{T_2}_2)^2\frac{1-\beta_1}{1+\beta_1}(1-\beta_1^{2T_1})\sigma^2\\
&\quad + (\beta^{r_k}_{n_k+1}\beta^{T_{n_k}}_{n_k}...\beta^{T_3}_3)^2\frac{1-\beta_2}{1+\beta_2}(1-\beta_2^{2T_2})\sigma^2\\
&\quad+ \dots\\
&\quad+(\beta^{r_k}_{n_k+1})^2\frac{1-\beta_{n_k}}{1+\beta_{n_k}}(1-\beta_{n_k}^{2T_{n_k}})\sigma^2\\
&\quad +\frac{1-\beta_{n_k+1}}{1+\beta_{n_k+1}}(1-\beta^{2r_k}_{n_k+1})\sigma^2.
\end{align*}
Since for any $l\in[1,n]$, we have
\begin{align*}
(\beta^{T_l}_l)^2&\leq \frac{1}{2},\\
1-\beta_l&\leq 1 - \beta_1,\\
1+\beta_l&\geq \frac{3}{2},\\
1-\beta_l^{2T_l}&< 1.
\end{align*}
Therefore,
\begin{align*}
\frac{1}{1-\beta_1}\EE[\|m^k - \sum_{i=1}^k b_{k,i} g^i\|^2]
&\leq (\beta^{r_k}_{n_k+1})^2(\frac{1}{2})^{n_k-1}\frac{2}{3}\cdot \frac{1-\beta_1}{1-\beta_1}\sigma^2\\
&\quad + (\beta^{r_k}_{n_k+1})^2(\frac{1}{2})^{n_k-2}\frac{2}{3}\cdot \frac{1-\beta_2}{1-\beta_1}\sigma^2\\
&\quad+ \dots\\
&\quad+(\beta^{r_k}_{n_k+1})^2(\frac{1}{2})^{0}\frac{2}{3}\cdot \frac{1-\beta_{n_k}}{1-\beta_1}\sigma^2\\
&\quad +\frac{2}{3} \cdot \frac{1-\beta_{n_k+1}}{1-\beta_1}\sigma^2\\
&\leq 2\sigma^2.
\end{align*}

\end{proof}

\begin{lemma}
\label{lem: m^k-1 for multistage}
Under the assumptions of Theorem \ref{thm: nonconvex Multistage}, the update vector $m^{k-1}$ in Algorithm \ref{alg: Multistage SGDM} satisfies
\begin{align*}
\frac{1}{1-\beta(k)}\EE[\|m^{k-1} - \sum_{i=1}^{k-1} b_{k-1,i} g^i\|^2] &\leq 24\frac{\beta_1}{\sqrt{\beta_n+\beta^2_n}}\sigma^2.
\end{align*}
\end{lemma}
\begin{proof}
By setting $k-1=T_1+ \dots +T_{n_{k-1}}+r_{k-1}$, we have
\begin{align*}
\EE[\|m^{k-1} - \sum_{i=1}^{k-1} b_{k-1,i} g^i\|^2]
&\leq (\beta^{r_{k-1}}_{n_{k-1}+1}\beta^{T_{n_{k-1}}}_{n_{k-1}}...\beta^{T_2}_2)^2\frac{1-\beta_1}{1+\beta_1}(1-\beta_1^{2T_1})\sigma^2\\
&\quad + (\beta^{r_{k-1}}_{n_{k-1}+1}\beta^{T_{n_{k-1}}}_{n_{k-1}}...\beta^{T_3}_3)^2\frac{1-\beta_2}{1+\beta_2}(1-\beta_2^{2T_2})\sigma^2\\
&\quad+ \dots\\
&\quad+(\beta^{r_{k-1}}_{n_{k-1}+1})^2\frac{1-\beta_{n_{k-1}}}{1+\beta_{n_{k-1}}}(1-\beta_{n_{k-1}}^{2T_{n_{k-1}}})\sigma^2\\
&\quad +\frac{1-\beta_{n_{k-1}+1}}{1+\beta_{n_{k-1}+1}}(1-\beta^{2r_{k-1}}_{n_{k-1}+1})\sigma^2.
\end{align*}

Similar as before, we have
\begin{align*}
\frac{1}{1-\beta(k)}\EE[\|m^{k-1} - \sum_{i=1}^k b_{k-1,i} g^i\|^2]
&\leq (\beta^{r_{k-1}}_{n_{k-1}+1})^2(\frac{1}{2})^{n_{k-1}-1}\frac{2}{3}\cdot \frac{1-\beta_1}{1-\beta(k)}\sigma^2\\
&\quad + (\beta^{r_{k-1}}_{n_{k-1}+1})^2(\frac{1}{2})^{n_{k-1}-2}\frac{2}{3}\cdot \frac{1-\beta_1}{1-\beta(k)}\sigma^2\\
&\quad+ \dots\\
&\quad+(\beta^{r_{k-1}}_{n_{k-1}+1})^2(\frac{1}{2})^{0}\frac{2}{3}\cdot \frac{1-\beta_1}{1-\beta(k)}\sigma^2\\
&\quad +\frac{2}{3}\cdot \frac{1-\beta_1}{1-\beta(k)}\sigma^2\\
&\leq 2\frac{1-\beta_1}{1-\beta(k)}\sigma^2.
\end{align*}

Finally, by applying 
\[
\frac{1-\beta_1}{1-\beta_n}\leq 12 \frac{\beta_1}{\sqrt{\beta_n+\beta^2_n}},
\]
we arrive at 
\begin{align*}
&\frac{1}{1-\beta(k)}\EE[\|m^{k-1} - \sum_{i=1}^k b_{k-1,i} g^i\|^2]\leq 24 \frac{\beta_1}{\sqrt{\beta_n+\beta^2_n}}\sigma^2.
\end{align*}
\end{proof}

\subsection{Generalization of Lemma \ref{lem: z^k update} for Multistage SGDM}

\begin{lemma}
\label{lem: z^k new update}
$z^k$ defined in \eqref{equ: z^k new} satisfies
\begin{align*}
z^{k+1}-z^k&=-\alpha(k)\tg^k,
\end{align*}
where $\alpha(k)$ is the stepsize applied at the $k$th step.
\end{lemma}
\begin{proof}
Recall that the auxiliary sequence $z^k$ is defined by 
\[
z^k = x^k - A_1 m^{k-1},
\]
where $A_1\equiv \frac{\alpha_i\beta_i}{1-\beta_i}$ and $\alpha_i, \beta_i$ are the stepsize and momentum weight at the $i$th stage, respectively. Therefore, we also have
\[
A_1\equiv \frac{\alpha(k)\beta(k)}{1-\beta(k)},
\]
where $\alpha(k), \beta(k)$ are the stepsize and momentum weight applied at the $k$th step. Using this, we obtain
\begin{align*}
z^{k+1}-z^k&=x^{k+1}-x^k-A_1(m^k-m^{k-1})\\
&=-\alpha(k)m^k-A_1(1-\beta(k))(\tg^k-m^{k-1})\\
&=-\alpha(k)m^k-\alpha(k)\beta(k)(\tg^k-m^{k-1})\\
&=\alpha(k)(\beta(k)m^{k-1}-m^k)-\alpha(k)\beta(k)\tg^k\\
&=-\alpha(k)\tg^k.
\end{align*}
\end{proof}

\subsection{Generalization of Lemma \ref{lem: difference} for Multistage SGDM}
\begin{lemma}
\label{lem: difference for multistage}
In Multistage SGDM(Algorithm \ref{alg: Multistage SGDM}), assume that the momentum weights at $n$ stages satisfy $\beta_1\leq\beta_2\leq...\leq \beta_n$. Then, we have 
\begin{align*}
\E\left [\left \|\frac{1}{1-\prod_{i=1}^k\beta(i)}\sum_{i=1}^k b_{k,i} g^i - g^k\right \|^2\right ]\leq \sum_{i=1}^{k-1}a_{k,i}\E[\|x^{j+1}-x^j\|^2],
\end{align*}
where $b_{k,i}=\big(1-\beta(i)\big)\prod_{j=i+1}^k \beta(j)$ and $\beta(i)$ is the momentum weight applied at the $i$th iteration, and
\begin{align}
\label{equ: a k i}
a_{k,i}=\frac{L^2\beta^{k-i}(k)}{1-\prod_{i=1}^k\beta(i)}\left(k-i+\frac{\beta(k)}{1-\beta(k)}\right).
\end{align}
\end{lemma}

\begin{proof}

By By \eqref{equ: m^k multistage}, \eqref{equ: b k i} and \eqref{equ: sum of weights}, we can compute that
\begin{align*}
&\E\left [\left \|\frac{1}{1-\prod_{i=1}^k\beta(i)}\sum_{i=1}^k b_{k,i} g^i - g^k\right \|^2\right ]\\
&=\E\|\frac{1}{1-\prod_{j=1}^k\beta(j)}\sum_{i=1}^k b_{k,i}(g^i-g^k)\|^2\\
&=\left(\frac{1}{1-\prod_{j=1}^k\beta(j)}\right)^2\sum_{i,j=1}^k b_{k,i} b_{k,j} \E\langle (g^k-g^i), (g^k-g^j)\rangle\\
&\leq \left(\frac{1}{1-\prod_{j=1}^k\beta(j)}\right)^2\sum_{i,j=1}^k b_{k,i} b_{k,j}(\frac{1}{2}\E\|g^k-g^i\|^2\\
&\quad\quad\quad\quad\quad\quad\quad\quad\quad\quad\quad\qquad\qquad+\frac{1}{2}\E\|g^k-g^j\|^2)\\
&=\left(\frac{1}{1-\prod_{j=1}^k\beta(j)}\right)\sum_{j=1}^k b_{k,j}\E\|g^k-g^j\|^2\\
&\leq\left(\frac{1}{1-\prod_{j=1}^k\beta(j)}\right)\sum_{j=1}^k b_{k,j}(k-j) \sum_{i=j}^{k-1}L^2\E\|x^{i+1}-x^i\|^2,
\end{align*}
where we have used \eqref{equ: sum of weights} in the first and third equality, and Cauchy-Schwarz in the first inequality. In the last inequality, we have applied the triangle inequality and the $L-$smoothness of $f$.

Consequently, we have
\begin{align*}
&\E\left\|\frac{1}{1-\prod_{j=1}^k\beta(j)}\sum_{i=1}^k b_{k,i} g^i-g^k\right\|^2\\
&\leq\left(\frac{1}{1-\prod_{j=1}^k\beta(j)}\right)\sum_{j=1}^k b_{k,j}(k-j) \sum_{i=j}^{k-1}L^2\E\|x^{i+1}-x^i\|^2\\
&=\left(\frac{1}{1-\prod_{j=1}^k\beta(j)}\right)\sum_{i=1}^{k-1}\sum_{j=1}^i b_{k,j}(k-j) L^2\E\|x^{i+1}-x^i\|^2\\
&= \sum_{i=1}^{k-1} d_{k,i}\E[\|x^{i+1}-x^i\|^2],
\end{align*}
where in the last step we have defined 
\begin{align}
\label{equ: d k i}
d_{k,i} = \left(\frac{L^2}{1-\prod_{j=1}^k\beta(j)}\right)\sum_{j=1}^i (k-j)b_{k,j}.
\end{align}
In the Proposition \ref{prop: a k i d k i} below, we shall see that $d_{k,i}\leq a_{k,i}$ for all $i\leq k-1$, where $a_{k,i}$ is defined in \eqref{equ: a k i}. Therefore, 
\begin{align*}
&\E[\|\frac{1}{1-\prod_{i=1}^k\beta(i)}\sum_{i=1}^k b_{k,i} g^i - g^k\|^2]\leq \sum_{i=1}^{k-1}a_{k,i}\E[\|x^{j+1}-x^j\|^2],
\end{align*}
and the proof will be complete.

\end{proof}

\begin{proposition}
\label{prop: a k i d k i}
$d_{k,i}$ defined in \eqref{equ: d k i} and $a_{k,i}$ defined in \eqref{equ: a k i} satisfy
\[
d_{k,i}\leq a_{k,i} \quad\text{for all}\quad i\leq k-1.
\]
\end{proposition}
\begin{proof}
We aim to show that $d_{k,i}\leq a_{k,i}$ for all $i\leq k-1$. Or equivalently, $d_{k,j}\leq a_{k,j}$ for all $j\leq k-1$.

In order to show $d_{k,j}\leq a_{k,j}$, we just need to show that
\begin{align}
\label{equ: what to show}
\sum_{i=1}^j (k-i)b_{k,i}\leq \beta^{k-j}(k)\left(k-j+\frac{\beta(k)}{1-\beta(k)}\right),
\end{align}
where 
\[
b_{k,i} = \big(1-\beta(i)\big)\prod_{j=i+1}^k \beta(j).
\]

Let $k=T_1+T_2+ \dots +T_{n_k}+r_k$, where $0\leq n_k\leq n-1$. If $n_{k}<n-1$, then $0\leq r_k\leq T_{n_k+1}-1$. If $n_k=n-1$, then $0\leq r_k\leq T_{n_k+1}=T_n$. 

Since $j\leq k-1$, we have $j=T_1+ \dots +T_{n_j}+r_j$, where $0\leq n_j\leq n_k$. 

Now, let us compute the left hand side of \eqref{equ: what to show} explicitly.
\begin{align*}
&\sum_{i=1}^j (k-i)b_{k,i}\\
&=\left(\sum_{i=1}^{T_1}+\sum_{i=T_1+1}^{T_1+T_2}+ \dots +\sum_{i=T_1+ \dots +T_{n_j}+1}^{T_1+\dots+T_{n_j}+r_j}\right)(k-i)b_{k,i}.
\end{align*}
Notice that
\begin{align*}
b_{k,i}=
\begin{cases}
\beta_{n_k+1}^{r_k}\beta_{n_k}^{T_{n_k}}\cdots \beta_2^{T_2}(1-\beta_1)\beta_1^{T_1-i},\qquad \quad  1\leq i \leq T_1,\\
\beta_{n_k+1}^{r_k}\beta_{n_k}^{T_{n_k}}\cdots \beta_3^{T_3}(1-\beta_2)\beta_1^{T_1+T_2-i}, \\
\qquad\qquad\qquad\qquad\qquad\qquad T_1+1\leq i \leq T_1+T_2,\\
\dots \dots \\
(1-\beta_{n_k+1})\beta_1^{T_1+ \dots +T_{n_k}+r_k-i},\\
\qquad\qquad\qquad\qquad \sum_{l=1}^{n_k}T_l+1\leq i \leq \sum_{l=1}^{n_k}T_l+r_k.
\end{cases}
\end{align*}
As a result, we have
\begin{align}
\label{equ: 111}
\begin{split}
&\sum_{i=1}^j (k-i)b_{k,i}\\
&=\left(\sum_{i=1}^{T_1}+\sum_{i=T_1+1}^{T_1+T_2}+ \dots +\sum_{i=T_1+ \dots +T_{n_j}+1}^{T_1+ \dots +T_{n_j}+r_j}\right)(k-i)b_{k,i}\\
&\leq \beta_{n_k+1}^{r_k}\beta_{n_k}^{T_{n_k}}\cdots\beta_2^{T_2}(1-\beta_1)\sum_{i=1}^{T_1}\beta_1^{T_1-i}(k-i)\\
&\quad + \beta_{n_k+1}^{r_k}\beta_{n_k}^{T_{n_k}}\cdots\beta_3^{T_3}(1-\beta_2)\sum_{i=T_1+1}^{T_1+T_2}\beta_2^{T_1+T_2-i}(k-i)\\
&\quad + \dots\\
&\quad + \beta_{n_k+1}^{r_k}\beta_{n_k}^{T_{n_k}}...\beta_{n_j+1}^{T_{n_j+1}}(1-\beta_{n_j})\\
&\qquad\qquad\sum_{i=T_1+ \dots +T_{n_j-1}+1}^{T_1+ \dots +T_{n_j}}\beta_{n_l}^{T_1+ \dots +T_{n_j}-i}(k-i)\\
&\quad + \beta_{n_k+1}^{r_k}\beta_{n_k}^{T_{n_k}}...\beta_{n_j+2}^{T_{n_j}+2}(1-\beta_{n_j+1})\\
&\quad\quad\quad\quad\sum_{i=T_1+ \dots +T_{n_j}+1}^{T_1+ \dots +T_{n_j}+r_j}\beta_{n_j+1}^{T_1+ \dots +T_{n_j}+r_j-i}(k-i),
\end{split}
\end{align}
where we have applied $r_j\leq T_{n_j+1}$ if $n_j<n_k$ and $r_j\leq r_k$ if $n_j=n_k$ in the last term. Since
\begin{align*}
\sum_{i=1}^l \beta^{k-i}(k-i) &= \beta^k\left(-\frac{k-1}{1-\beta}-\frac{1}{(1-\beta)^2}\right)\\
&\quad + \beta^{k-l}\left(\frac{k-l}{1-\beta}+\frac{\beta}{(1-\beta)^2}\right).
\end{align*}
we have
\begin{align*}
\sum_{i=1}^{T_1}\beta_1^{T_1-i}(k-i)&= \beta_1^{T_1-k}\sum_{i=1}^{T_1}\beta_1^{k-i}(k-i)\\
&=\beta_1^{T_1}\left(-\frac{k-1}{1-\beta_1}-\frac{1}{(1-\beta_1)^2}\right) \\
&\,\,\, + \left(\frac{k-T_1}{1-\beta_1}+\frac{\beta_1}{(1-\beta_1)^2}\right),\\
\sum_{i=T_1+1}^{T_1+T_2}\beta_2^{T_1+T_2-i}(k-i)&=\sum_{i=1}^{T_2} \beta_2^{T_2-i}(k-T_1-i)\\
&=\beta_2^{T_1+T_2-k}\sum_{i=1}^{T_2}\beta_2^{k-T_1-i}(k-T_1-i)\\
&=\beta_2^{T_2}\left(-\frac{k-T_1-1}{1-\beta_2}-\frac{1}{(1-\beta_2)^2}\right)\\
&\,\,\, + \left(\frac{k-T_1-T_2}{1-\beta_2}+\frac{\beta_2}{(1-\beta_2)^2}\right).
\end{align*}
And that in general
\begin{align*}
&\sum_{i=T_1+ \dots +T_{n_j}+1}^{T_1+ \dots +T_{n_j}+r_j}\beta_{n_j+1}^{T_1+ \dots +T_{n_j}+r_j-i}(k-i)\\
&=\sum_{i=1}^{r_j}\beta_{n_j+1}^{r_j-i}(k-T_1- \dots - T_{n_j}-i)\\
&=\beta_{n_j+1}^{T_1+ \dots +T_{n_j}+r_j-k}\sum_{i=1}^{r_j}\beta_{n_j+1}^{k-T_1- \dots - T_{n_j}-i}(k-T_1- \dots - T_{n_j}-i)\\
&=\beta_{n_j+1}^{r_j}\left(-\frac{k-T_1- \dots - T_{n_j}-1}{1-\beta_{n_j+1}}-\frac{1}{(1-\beta_{n_j+1})^2}\right)\\
&\quad +\left(\frac{k-T_1- \dots - T_{n_j}-r_j}{1-\beta_{n_j+1}}+\frac{\beta_{n_j+1}}{(1-\beta_{n_j+1})^2}\right).
\end{align*}
By applying these equalities on \eqref{equ: 111},  we have
\begin{align*}
&\sum_{i=1}^j (k-i)b_{k,i}\\
&= \beta_{n_k+1}^{r_k}\beta_{n_k}^{T_{n_k}}...\beta_2^{T_2}\Bigg(\beta_1^{T_1}\left(-(k-1)-\frac{1}{1-\beta_1}\right) + \left((k-T_1)+\frac{\beta_1}{1-\beta_1}\right)\Bigg)\\
&\quad + \beta_{n_k+1}^{r_k}\beta_{n_k}^{T_{n_k}}...\beta_3^{T_3}\Bigg(\beta_2^{T_2}\bigg(-(k-T_1-1)-\frac{1}{1-\beta_2}\bigg)\\
&\qquad\qquad\qquad\qquad\qquad + \bigg((k-T_1-T_2)+\frac{\beta_2}{1-\beta_2}\bigg)\Bigg)\\
&\quad + \dots\\
&\quad + \beta_{n_k+1}^{r_k}\beta_{n_k}^{T_{n_k}}...\beta_{n_j+1}^{T_{n_j+1}}\\
&\qquad\qquad\Bigg(\beta_{n_j}^{T_{n_j}}\bigg(-(k-T_1- \dots - T_{n_j-1}-1)-\frac{1}{1-\beta_{n_j}}\bigg)\\
&\qquad\qquad\qquad + \bigg((k-T_1- \dots - T_{n_j})+\frac{\beta_{n_j}}{1-\beta_{n_j}}\bigg)\Bigg)\\
&\quad + \beta_{n_k+1}^{r_k}\beta_{n_k}^{T_{n_k}}...\beta_{n_j+2}^{T_{n_j+2}}\\
&\qquad\qquad\Bigg(\beta_{n_j+1}^{r_j}\bigg(-(k-T_1- \dots - T_{n_j}-1)-\frac{1}{1-\beta_{n_j+1}}\bigg)\\
&\qquad\qquad\qquad+ \bigg((k-T_1-T_{n_j}-r_j)+\frac{\beta_{n_j+1}}{1-\beta_{n_j+1}}\bigg)\Bigg).
\end{align*}
This yields
\begin{align*}
\sum_{i=1}^j (k-i)b_{k,i}
&=\beta_{n_k+1}^{r_k}\beta_{n_k}^{T_{n_k}}...\beta_2^{T_2}\beta_1^{T_1}\bigg(-(k-1)-\frac{1}{1-\beta_1}\bigg)\\
&\quad + \beta_{n_k+1}^{r_k}\beta_{n_k}^{T_{n_k}}...\beta_2^{T_2}\bigg(\frac{\beta_1}{1-\beta_1}+1-\frac{1}{1-\beta_2}\bigg)\\
&\quad + \beta_{n_k+1}^{r_k}\beta_{n_k}^{T_{n_k}}...\beta_3^{T_3}\bigg(\frac{\beta_2}{1-\beta_2}+1-\frac{1}{1-\beta_3}\bigg)\\
&\quad + \dots\\
&\quad +\beta_{n_k+1}^{r_k}\beta_{n_k}^{T_{n_k}}...\beta_{n_j}^{T_{n_j}}\bigg(\frac{\beta_{n_j-1}}{1-\beta_{n_j-1}}+1-\frac{1}{1-\beta_{n_j}}\bigg)\\
&\quad +\beta_{n_k+1}^{r_k}\beta_{n_k}^{T_{n_k}}...\beta_{n_j+1}^{T_{n_j+1}}\bigg(k-T_1- \dots - T_{n_j}+\frac{\beta_{n_j}}{1-\beta_{n_j}}\bigg)\\
&\quad + \beta_{n_k+1}^{r_k}\beta_{n_k}^{T_{n_k}}...\beta_{n_j+2}^{T_{n_j+2}} \\
&\qquad \cdot \Bigg(\beta_{n_j+1}^{r_j}\bigg(-(k-T_1- \dots - T_{n_j}-1)-\frac{1}{1-\beta_{n_j+1}}\bigg)\\
&\qquad\qquad\qquad+ \bigg((k-T_1-T_{n_j}-r_j)+\frac{\beta_{n_j+1}}{1-\beta_{n_j+1}}\bigg)\Bigg).
\end{align*}
On the right hand side, the first $n_j$ terms are non-positive since $\beta_1\leq\beta_2\leq...\leq\beta_n$. Therefore,
\begin{align*}
\sum_{i=1}^j (k-i)b_{k,i}
&\leq\beta_{n_k+1}^{r_k}\beta_{n_k}^{T_{n_k}}...\beta_{n_j+1}^{T_{n_j+1}}\bigg(k-T_1- \dots - T_{n_j}+\frac{\beta_{n_j}}{1-\beta_{n_j}}\bigg)\\
&\quad + \beta_{n_k+1}^{r_k}\beta_{n_k}^{T_{n_k}}...\beta_{n_j+2}^{T_{n_j+2}}\\
&\qquad\qquad\Bigg(\beta_{n_j+1}^{r_j}\bigg(-(k-T_1- \dots - T_{n_j}-1)-\frac{1}{1-\beta_{n_j+1}}\bigg)\\
&\qquad\qquad\qquad+ \bigg((k-T_1-T_{n_j}-r_j)+\frac{\beta_{n_j+1}}{1-\beta_{n_j+1}}\bigg)\Bigg).
\end{align*}

By applying $\beta_{n_j+1}^{r_j}\geq \beta^{T_{n_j}+1}_{n_j+1}$ and $k-T_1- \dots - T_{n_j}-1=k-(j-r_j)-1\geq 0$ (since $j\leq k-1$), we arrive at
\begin{align*}
\sum_{i=1}^j (k-i)b_{k,i}
&\leq\beta_{n_k+1}^{r_k}\beta_{n_k}^{T_{n_k}}...\beta_{n_j+1}^{T_{n_j+1}}\bigg(k-T_1- \dots - T_{n_j}+\frac{\beta_{n_j}}{1-\beta_{n_j}}\bigg)\\
&\quad + \beta_{n_k+1}^{r_k}\beta_{n_k}^{T_{n_k}}...\beta_{n_j+2}^{T_{n_j+2}}\\
&\qquad\qquad\Bigg(\beta_{n_j+1}^{T_{n_j}+1}\bigg(-(k-T_1- \dots - T_{n_j}-1)-\frac{1}{1-\beta_{n_j+1}}\bigg)\\
&\qquad\qquad\qquad+ \bigg((k-T_1-T_{n_j}-r_j)+\frac{\beta_{n_j+1}}{1-\beta_{n_j+1}}\bigg)\Bigg)\\
&\leq \beta_{n_k+1}^{r_k}\beta_{n_k}^{T_{n_k}}...\beta_{n_j+1}^{T_{n_j+1}}\bigg(\frac{\beta_{n_j}}{1-\beta_{n_j}}+1-\frac{1}{1-\beta_{n_j+1}}\bigg)\\
&\quad + \beta_{n_k+1}^{r_k}\beta_{n_k}^{T_{n_k}}...\beta_{n_j+2}^{T_{n_j+2}}\bigg(k-T_1- \dots - T_{n_j}-r_j+\frac{\beta_{n_j+1}}{1-\beta_{n_j+1}}\bigg)\\
&\leq \beta_{n_k+1}^{r_k}\beta_{n_k}^{T_{n_k}}...\beta_{n_j+2}^{T_{n_j+2}}\bigg(k-T_1- \dots - T_{n_j}-r_j+\frac{\beta_{n_j+1}}{1-\beta_{n_j+1}}\bigg)\\
&=\beta_{n_k+1}^{r_k}\beta_{n_k}^{T_{n_k}}...\beta_{n_j+2}^{T_{n_j+2}}\bigg(k-j+\frac{\beta_{n_j+1}}{1-\beta_{n_j+1}}\bigg).
\end{align*}
Now let us consider two cases: $r_k>0$ and $r_k=0$.
\begin{enumerate}
\item $r_k>0$.

In this case, we apply $\beta_1\leq...\leq \beta_n$ to get
\begin{align*}
&\sum_{i=1}^j (k-i)b_{k,i}\\
&\leq \beta^{r_k+T_{n_k}+ \dots +T_{n_j+2}}_{n_k+1}\bigg(k-j+\frac{\beta_{n_k+1}}{1-\beta_{n_k+1}}\bigg).
\end{align*}
Notice that 
\begin{align*}
r_k+T_{n_k}+ \dots +T_{n_j+2}&= (T_1+ \dots +T_{n_k}+r_k)\\
&\quad -(T_1+ \dots +T_{n_j}+T_{n_j+1})\\
&\leq (T_1+ \dots +T_{n_k}+r_k)\\
&\quad -(T_1+ \dots +T_{n_j}+r_j)\\
&=k-j.
\end{align*}
This tells us that
\begin{align*}
&\sum_{i=1}^j (k-i)b_{k,i}\leq \beta^{k-j}_{n_k+1}\bigg(k-j+\frac{\beta_{n_k+1}}{1-\beta_{n_k+1}}\bigg).
\end{align*}
Since $r_k>0$, iteration $k$ is at the $(n_k+1)-$th stage, we have $\beta(k)=\beta_{n_k+1}$, and the above inequality is exactly what we want to show in \eqref{equ: what to show}.

\item $r_k=0$

In this case, we apply $\beta_1\leq...\leq \beta_n$ to get
\begin{align*}
&\sum_{i=1}^j (k-i)b_{k,i} \leq \beta^{T_{n_k}+ \dots +T_{n_j+2}}_{n_k}\bigg(k-j+\frac{\beta_{n_j+1}}{1-\beta_{n_j+1}}\bigg).
\end{align*}
Notice that 
\begin{align*}
r_k+T_{n_k}+ \dots +T_{n_j+2}&= (T_1+ \dots +T_{n_k}+r_k)\\
&\quad -(T_1+ \dots +T_{n_j}+T_{n_j+1})\\
&\leq (T_1+ \dots +T_{n_k}+r_k)\\
&\quad -(T_1+ \dots +T_{n_j}+r_j)\\
&=k-j.
\end{align*}
This tells us that
\begin{align*}
&\sum_{i=1}^j (k-i)b_{k,i}\leq \beta^{k-j}_{n_k}\bigg(k-j+\frac{\beta_{n_j+1}}{1-\beta_{n_j+1}}\bigg),
\end{align*}
Since $r_k=0$, we have $\beta(k)=\beta_{n_k}$ and by $j\leq k-1$ we deduce that $n_j\leq n_k-1$ (Otherwise $j=T_1+ \dots +T_{n_j}+r_j=T_1+ \dots +T_{n_k}+r_j\geq T_1+ \dots +T_{n_k}=k$). Therefore, we have
\begin{align*}
&\sum_{i=1}^j (k-i)b_{k,i}\leq \beta^{k-j}_{n_k}\bigg(k-j+\frac{\beta_{n_k}}{1-\beta_{n_k}}\bigg),
\end{align*}
which is exactly what we want to show in \eqref{equ: what to show}.

\end{enumerate}

\end{proof}

\section{Main Theory for Multistage SGDM}
\label{app: proof for Multistage SGDM}
In this section, we prove the main convergence theory of Multistage SGDM.

\subsection{Proof of Proposition \ref{prop: L^k for Multistage}}
Proposition \ref{prop: L^k for Multistage} is a generalization of Propositions \ref{prop: f(z^k)} and \ref{prop: L^k} to the multistage case. Therefore, its proof is similar to those of  Propositions \ref{prop: f(z^k)} and \ref{prop: L^k}.

First of all, by the smoothness of $f$ we have
\begin{align}
\label{equ: f smoothness 1}
\begin{split}
\E[f(z^{k+1})]&\leq \E [f(z^k)] + \E \langle \nabla f(z^k), z^{k+1}-z^k \rangle +\frac{L}{2}\E\|z^{k+1}-z^k\|^2\\
&=\E [f(z^k)] + \E \langle \nabla f(z^k), -\alpha(k) \tg^k \rangle +\frac{L\alpha^2(k)}{2}\E\|\tg^k\|^2,
\end{split}
\end{align}
where we have applied Lemma \ref{lem: z^k new update} in the second step. Note that $\alpha(k)$ is the stepsize applied at the $k-$th iteration.

For the inner product term, we have
\begin{align*}
\E \langle \nabla f(z^k), -\alpha(k) \tg^k \rangle
&=\E \langle \nabla f(z^k), -\alpha(k) g^k \rangle,
\end{align*}
which follows from the fact that $z^k$ is determined by the previous $k-1$ random samples $\zeta^1, \zeta^2,...\zeta^{k-1}$, which is independent of $\zeta^k$, and $\EE_{\zeta^k} [\tg^k]=g^k$.

As a result, we can write
\begin{align}
\label{equ: inner product new}
\begin{split}
\E \langle \nabla f(z^k), -\alpha(k) \tg^k \rangle
&=\E \langle \nabla f(z^k)-g^k, -\alpha(k) g^k \rangle - \alpha(k) \E\|g^k\|^2\\
&\leq \alpha(k) \frac{\rho_{0,k}}{2}L^2\E[\|z^k-x^k\|^2] +\alpha(k)\frac{1}{2\rho_{0,k}}\E[\|g^k\|^2]-\alpha(k) \E[\|g^k\|^2],
\end{split}
\end{align}
where $\rho_{0,k}>0$ can be any positive constant. 

Combining \eqref{equ: f smoothness 1} and \eqref{equ: inner product new} gives
\begin{align*}
\E[f(z^{k+1})]&\leq \E[f(z^k)]+\alpha(k)\frac{\rho_{0,k}}{2}L^2\E[\|z^k-x^k\|^2]\\
&\quad+\big(\alpha(k)\frac{1}{2\rho_{0,k}}-\alpha(k)\big)\E[\|g^k\|^2] +\frac{L\alpha^2(k)}{2}\E[\|\tg^k\|^2]
\end{align*}
By \eqref{equ: z^k new} we know that $z^k-x^k=-A_1m^{k-1}$, which leads to
\begin{align*}
\begin{split}
    \E[f(z^{k+1})]&\leq \E[f(z^k)]+\alpha(k)\frac{\rho_{0,k}}{2}L^2A_1^2\E[\|m^{k-1}\|^2]\\
&\quad+\big(\alpha(k)\frac{1}{2\rho_{0,k}}-\alpha(k)\big)\E[\|g^k\|^2] +\frac{L\alpha^2(k)}{2}(\sigma^2+\E[\|g^k\|^2]).
\end{split}
\end{align*}
Therefore, we have 
\begin{align}
\label{equ: L^k intermediate new}
\begin{split}
\E[L^{k+1}-L^k]
&\leq \alpha(k) \frac{\rho_{0,k}}{2}L^2A_1^2\E[\|m^{k-1}\|^2]\\
&\quad + \left(\alpha(k)\frac{1}{2\rho_{0,k}}-\alpha(k)+\frac{L\alpha^2(k)}{2}\right)\E[\|g^k\|^2]+\frac{L\alpha^2(k)}{2}\sigma^2\\
&\quad +c_1\alpha^2(k)\E[\|m^k\|^2]\\
&\quad +\sum_{i=1}^{k-1}(c_{i+1}-c_i)\E[\|x^{k+1-i}-x^{k-i}\|^2]\\
&\leq \alpha(k) \frac{\rho_{0,k}}{2}L^2A_1^2\bigg( 2\EE[\|m^{k-1} - \sum_{i=1}^{k-1} b_{k-1,i} g^i\|^2] + 2 \EE[\|\sum_{i=1}^{k-1} b_{k-1,i} g^i\|^2] \bigg)\\
&\quad + \left(\alpha(k)\frac{1}{2\rho_{0,k}}-\alpha(k)+\frac{L\alpha^2(k)}{2}\right)\E[\|g^k\|^2]+\frac{L\alpha^2(k)}{2}\sigma^2\\
&\quad +c_1\alpha^2(k)\bigg(2\EE[\|m^{k} - \sum_{i=1}^{k} b_{k,i} g^i\|^2] + 2 \EE[\|\sum_{i=1}^{k} b_{k,i} g^i\|^2]\bigg)\\
&\quad +\sum_{i=1}^{k-1}(c_{i+1}-c_i)\E[\|x^{k+1-i}-x^{k-i}\|^2].
\end{split}
\end{align}
On the other hand, we know that
\begin{align}
\label{equ: useful inequalities new 1}
\begin{split}
\E[\|\frac{1}{1-\prod_{i=1}^{k}\beta(i)}\sum_{i=1}^{k} b_{k,i} g^i\|^2]
&=\E[\|\frac{1}{1-\prod_{i=1}^{k}\beta(i)}\sum_{i=1}^{k} b_{k,i} g^i-g^k+g^k\|^2]\\
&\leq 2\E[\|g^k\|^2]+2\E[\|\frac{1}{1-\prod_{i=1}^{k}\beta(i)}\sum_{i=1}^{k} b_{k,i} g^i-g^k\|^2].
\end{split}
\end{align}
Furthermore, 
\begin{align}
\label{equ: m^k-1 to m^k new}
\begin{split}
&\E[\|\frac{1}{1-\prod_{i=1}^k \beta(i)}\sum_{i=1}^{k} b_{k,i} g^i - g^k\|^2]\\
&=\E[\|\frac{1}{1-\prod_{i=1}^k \beta(i)}\beta(k)\sum_{i=1}^{k-1} b_{k-1,i} g^i+\frac{1-\beta(k)}{1-\prod_{i=1}^k \beta(i)} g^k-g^k\|^2]\\
&=\E[\|\frac{1}{1-\prod_{i=1}^k \beta(i)}\beta(k)\sum_{i=1}^{k-1} b_{k-1,i} g^i-\beta(k)\frac{1-\prod_{i=1}^{k-1}\beta(i)}{1-\prod_{i=1}^k \beta(i)} g^k\|^2]\\
&=\beta^2(k)\left(\frac{1-\prod_{i=1}^{k-1}\beta(i)}{1-\prod_{i=1}^k\beta(i)}\right)^2 \E [\|\frac{1}{1-\prod_{i=1}^{k-1}\beta(i)}\sum_{i=1}^{k-1} b_{k-1,i} g^i-g^k\|^2].
\end{split}
\end{align}
Therefore, we have
\begin{align}
\label{equ: useful inequalities new 2}
\begin{split}
&\E[\|\frac{1}{1-\prod_{i=1}^{k-1}\beta(i)}\sum_{i=1}^{k-1} b_{k-1,i} g^i\|^2]\\
&=\E[\|\frac{1}{1-\prod_{i=1}^{k-1}\beta(i)}\sum_{i=1}^{k-1} b_{k-1,i} g^i-g^k+g^k\|^2]\\
&\leq 2\E[\|g^k\|^2]+2\E[\|\frac{1}{1-\prod_{i=1}^{k-1}\beta(i)}\sum_{i=1}^{k-1} b_{k-1,i} g^i-g^k\|^2]\\
&=2\E[\|g^k\|^2]+2\frac{1}{\beta^2(k)}\left(\frac{1-\prod_{i=1}^{k}\beta(i)}{1-\prod_{i=1}^{k-1}\beta(i)}\right)^2\E[\|\frac{1}{1-\prod_{i=1}^{k}\beta(i)}\sum_{i=1}^{k} b_{k,i} g^i-g^k\|^2].
\end{split}
\end{align}
Plugging \eqref{equ: useful inequalities new 1} and \eqref{equ: useful inequalities new 2} into \eqref{equ: L^k intermediate new} gives us
\begin{align}
\label{equ: L^k intermediate new 1}
\begin{split}
&\E[L^{k+1}-L^k]\\
&\leq \bigg(-\alpha(k)+\alpha(k)\frac{1}{2\rho_{0,k}}+2\alpha(k)\rho_{0,k}L^2A_1^2+\frac{L\alpha^2(k)}{2}+4c_1\alpha^2(k)\bigg)\E[\|g^k\|^2]\\
&\quad +\left(\alpha(k){\rho_{0,k}}L^2A_1^2\EE[\|m^{k-1} - \sum_{i=1}^{k-1} b_{k-1,i} g^i\|^2])+\frac{1}{2}L\alpha^2(k)\sigma^2+2c_1\alpha^2(k)\EE[\|m^{k} - \sum_{i=1}^{k} b_{k,i} g^i\|^2]\right)\\
&\quad + \sum_{i=1}^{k-1}(c_{i+1}-c_i)\E[\|x^{k+1-i}-x^{k-i}\|^2]\\
&\quad +2\alpha(k)\rho_{0,k} L^2A_1^2\frac{1}{\beta^2(k)}\left(1-\prod_{i=1}^{k}\beta(i)\right)^2\E[\|\frac{1}{1-\prod_{i=1}^{k}\beta(i)} \sum_{i=1}^{k} b_{k,i} g^i-g^k\|^2]\\
&\quad +4c_1\alpha^2(k)\left(1-\prod_{i=1}^{k}\beta(i)\right)^2\E[\|\frac{1}{1-\prod_{i=1}^{k}\beta(i)}\sum_{i=1}^{k} b_{k,i} g^i-g^k\|^2]
\end{split}
\end{align}
In the rest of the proof, we will show that the sum of the last three terms in \eqref{equ: L^k intermediate new 1} is non-positive.

First, by Lemma \ref{lem: difference for multistage} we know that
\begin{align*}
\E\|\frac{1}{1-\prod_{i=1}^k\beta(i)}\sum_{i=1}^{k} b_{k,i} g^i-g^k\|^2\leq \sum_{i=1}^{k-1}a_{k,i}\E\|x^{i+1}-x^i\|^2,
\end{align*}
where 
\begin{align*}
a_{k, i}=\frac{L^2\beta^{k-i}(k)}{1-\prod_{i=1}^k\beta(i)}\left(k-i+\frac{\beta(k)}{1-\beta(k)}\right).
\end{align*}
Or equivalently, 
\begin{align*}
\E\|\frac{1}{1-\prod_{i=1}^k\beta(i)}\sum_{i=1}^{k} b_{k,i} g^i-g^k\|^2 \leq \sum_{i=1}^{k-1}a_{k,k-i}\E\|x^{k+1-i}-x^{k-i}\|^2,
\end{align*}
where 
\begin{align*}
a_{k, k-i}=\frac{L^2\beta^{i}(k)}{1-\prod_{i=1}^k\beta(i)}\left(i+\frac{\beta(k)}{1-\beta(k)}\right).
\end{align*}
Therefore, in order to make the sum of the last three terms of \eqref{equ: L^k intermediate new 1} to be non-positive, we need to enforce that
\begin{align*}
c_{i+1}&\leq c_i -\bigg(4c_1\alpha^2(k)\big(1-\prod_{i=1}^k\beta(i)\big)^2+2\alpha(k)\rho_{0,k}L^2A_1^2\frac{1}{\beta^2(k)}\big(1-\prod_{i=1}^k\beta(i)\big)^2\bigg)a_{k,k-i}
\end{align*}
for all $i\geq 1$ and $k\geq 1$.

Since $1-\prod_{i=1}^k\beta(i)<1$, $\beta_1\leq\beta(k)\leq\beta_n$, and $\alpha_1\leq\alpha(k)\leq\alpha_n$, we need to enforce the following for all $i\geq 1$:
\begin{align*}
c_{i+1}&\leq c_i -\bigg(4c_1\alpha^2_1+2\alpha(k)\rho_{0,k}L^2A_1^2\frac{1}{\beta^2_1}\bigg)\beta_n^i(i+\frac{\beta_n}{1-\beta_n})L^2.
\end{align*}
Recall that $\frac{\alpha_i\beta_i}{1-\beta_i}\equiv A_1$ for all $n$ stages $i=1,2,...,n$. This gives us
\begin{align*}
c_{i+1}&\leq c_i -\bigg(4c_1\alpha^2_1+2\alpha(k)\rho_{0,k}L^2\frac{\alpha_1^2}{(1-\beta_1)^2}\bigg)\beta_n^i(i+\frac{\beta_n}{1-\beta_n})L^2.
\end{align*}
Let us also set 
\begin{align}
\label{equ: rho 0 k choice}
    \rho_{0,k}=\frac{1-\beta(k)}{2L\alpha(k)}.
\end{align}
Then, we need to enforce
\begin{align*}
c_{i+1}&\leq c_i -\bigg(4c_1\alpha^2_1+2\frac{1-\beta(k)}{2}L\frac{\alpha_1^2}{(1-\beta_1)^2}\bigg)\beta_n^i(i+\frac{\beta_n}{1-\beta_n})L^2.
\end{align*}
Since $\beta_1\leq \beta(k)$, it suffices to enforce that
\begin{align}
\label{equ: c i choice new}
c_{i+1}&= c_i -\bigg(4c_1\alpha^2_1+ L\frac{\alpha_1^2}{(1-\beta_1)}\bigg)\beta_n^i(i+\frac{\beta_n}{1-\beta_n})L^2.
\end{align}
Note that the equalities in \eqref{equ: c i choice new} does not depend on $k$. In order for $c_i>0$ for all $i\geq 1$, we can determine $c_1$ by
\begin{align*}
c_1 = \left(4c_1\alpha_1^2+L\frac{\alpha_1^2}{(1-\beta_1)}\right)\sum_{i=1}^{\infty}\beta_n^i(i+\frac{\beta_n}{1-\beta_n})L^2.
\end{align*}
Since 
\[
\sum_{i=1}^j i \beta_n^i = \frac{1}{1-\beta_n}\left(\frac{\beta_n(1-\beta_n^j)}{1-\beta_n}-j\beta_n^{j+1}\right),
\]
we have $\sum_{i=1}^{\infty} i\beta_n^i = \frac{\beta_n}{(1-\beta_n)^2}$ and
\begin{align*}
c_1=\left(4c_1\alpha_1^2+L\frac{\alpha_1^2}{(1-\beta_1)}\right)\frac{\beta_n+\beta_n^2}{(1-\beta_n)^2}L^2.
\end{align*}
This stipulates that
\begin{align}
\label{equ: c 1 choice new}
c_1 = \frac{\frac{\alpha_1^2}{(1-\beta_1)}\frac{\beta_n+\beta_n^2}{(1-\beta_n)^2}L^3}{1-4\alpha_1^2\frac{\beta_n+\beta_n^2}{(1-\beta_n)^2}L^2}.
\end{align}
Notice that $A_1=\frac{1}{24\sqrt{2}L}$ and $\frac{1-\beta_1}{\beta_1}\leq 12\frac{1-\beta_{n}}{\sqrt{\beta_{n}+\beta^2_{n}}}$ ensures 
\[
4L^2\alpha_1^2\frac{\beta_n+\beta_n^2}{(1-\beta_n)^2}\leq \frac{1}{2}
\]
and therefore
\begin{align}
\label{equ: c 1 bound new}
    0<c_1\leq 2\frac{\alpha_1^2}{(1-\beta_1)}\frac{\beta_n+\beta_n^2}{(1-\beta_n)^2}L^3\leq \frac{L}{4(1-\beta_1)}.
\end{align}
With the choices of $c_i$ in \eqref{equ: c i choice new} and \eqref{equ: c 1 choice new}, the sum of the last three terms of \eqref{equ: L^k intermediate new 1} is non-positive. Therefore,

\begin{align}
\label{equ: L^k intermediate new 2}
\begin{split}
&\E[L^{k+1}-L^k]\\
&\leq \bigg(-\alpha(k)+\alpha(k)\frac{1}{2\rho_{0,k}}+2\alpha(k)\rho_{0,k}L^2A_1^2+\frac{L\alpha^2(k)}{2}+4c_1\alpha^2(k)\bigg)\E[\|g^k\|^2]\\
&\quad +\bigg(\alpha(k){\rho_{0,k}}L^2A_1^2\EE[\|m^{k-1} - \sum_{i=1}^{k-1} b_{k-1,i} g^i\|^2]+\frac{1}{2}L\alpha^2(k)\sigma^2+2c_1\alpha^2(k)\EE[\|m^{k} - \sum_{i=1}^{k} b_{k,i} g^i\|^2]\bigg).
\end{split}
\end{align}
Taking $\rho_{0,k} = \frac{1-\beta(k)}{2L\alpha(k)}$ in \eqref{equ: L^k intermediate new 2} gives
\begin{align*}
&\E[L^{k+1}-L^k]\\
&\leq \bigg(-\alpha(k)+\frac{3-\beta(k)+2\beta^2(k)}{2(1-\beta(k))}L\alpha^2(k) +4c_1\alpha^2(k)\bigg)\E[\|g^k\|^2]\\
&\quad +\bigg(\frac{\beta^2(k)}{2(1-\beta(k))}L\alpha^{2}(k)\EE[\|m^{k-1} - \sum_{i=1}^{k-1} b_{k-1,i} g^i\|^2]+\frac{1}{2}L\alpha^2(k)\sigma^2+2c_1\alpha^2(k)\EE[\|m^{k} - \sum_{i=1}^{k} b_{k,i} g^i\|^2]\bigg).
\end{align*}
Finally, by applying Lemma \ref{lem: m^k for multistage} and Lemma \ref{lem: m^k-1 for multistage}, we arrive at
\begin{align*}
&\E[L^{k+1}-L^k]\\
&\leq \bigg(-\alpha(k)+\frac{3-\beta(k)+2\beta^2(k)}{2(1-\beta(k))}L\alpha^2(k) +4c_1\alpha^2(k)\bigg)\E[\|g^k\|^2]\\
&\quad +\bigg(\frac{\beta^2(k)}{2}L\alpha^{2}(k)24 \frac{\beta_1}{\sqrt{\beta_n+\beta^2_n}}\sigma^2+\frac{1}{2}L\alpha^2(k)\sigma^2+4c_1(1-\beta_1)\alpha^2(k)\sigma^2\bigg).
\end{align*}

\subsection{Proof of Theorem \ref{thm: nonconvex Multistage}}

From \eqref{equ: L^k intermediate new 2} we know that
\begin{align}
\label{equ: L^k intermediate new 3}
\begin{split}
\E[L^{k+1}-L^k]\leq -R_{1,k}\E[\|g^k\|^2] +R_{2,k},
\end{split}
\end{align}
where
\begin{align}
    R_{1,k} &= \alpha(k)-\alpha(k)\frac{1}{2\rho_{0,k}}-2\alpha(k)\rho_{0,k}L^2A_1^2-\frac{L\alpha^2(k)}{2}-4c_1\alpha^2(k)\label{equ: R_1 new}\\
    R_{2,k} &= \alpha(k){\rho_{0,k}}L^2A_1^2\EE[\|m^{k-1} - \sum_{i=1}^{k-1} b_{k-1,i} g^i\|^2]+\frac{1}{2}L\alpha^2(k)\sigma^2+2c_1\alpha^2(k)\EE[\|m^{k} - \sum_{i=1}^{k} b_{k,i} g^i\|^2].\label{equ: R_2 new}
\end{align}
This immediately tells us that
\begin{align}
\label{equ: 2 new}
L^1\geq\E[L^1-L^{k+1}]\geq \sum_{i=1}^k R_{1,i}\E[\|g^i\|^2]-\sum_{i=1}^k R_{2,i},
\end{align}

In the rest the proof, we will bound $R_{1,i}$ and $R_{2,i}$ appropriately. 

First, let us show that $R_{1,i}\geq \frac{\alpha(i)}{2}$ under $\rho_{0,i} = \frac{1-\beta(i)}{2L\alpha(i)}$ as in \eqref{equ: rho 0 k choice} and $\alpha(i)=\frac{A_1(1-\beta(i))}{\beta(i)}=\frac{1-\beta(i)}{24\sqrt{2}L\beta(i)}$.

From \eqref{equ: c 1 bound new} we know that
\begin{align*}
c_1\leq \frac{L}{4(1-\beta_1)}.
\end{align*}
Therefore, in order for $R_{1,i}\geq \frac{\alpha(i)}{2}$, it suffices to have
\begin{align}
\label{equ: 1 new}
\begin{split}
&\alpha(i)\frac{1}{2\rho_{0,i}}+2\alpha(i) \rho_{0,i}L^2A_1^2+\frac{L\alpha^2(i)}{2}+4\frac{L}{4(1-\beta_1)}\alpha^2(i)\leq \frac{\alpha(i)}{2}.
\end{split}
\end{align}
By $\beta(i)\geq\beta_1\geq \frac{1}{2}$ we know that 
$$\alpha(i)=\frac{1-\beta(i)}{24\sqrt{2} L\beta(i)}\leq \frac{1}{2L}.$$ 
Therefore, $\frac{L\alpha^2(i)}{2}\leq \frac{\alpha(i)}{4}$. Furthermore, $\rho_{0,i}=\frac{1-\beta(i)}{2L\alpha(i)}$ yields
\begin{align*}
&\alpha(i)\frac{1}{2\rho_{0,i}}+2\alpha(i) \rho_{0,i}L^2A_1^2+4\frac{L}{4(1-\beta_1)}\alpha^2(i)\\
&= \frac{L\alpha^2(i)}{1-\beta(i)}+\frac{\beta^2(i)L\alpha^2(i)}{\big(1-\beta(i)\big)}+\frac{L}{(1-\beta_1)}\alpha^2(i)\\
&\leq \frac{\alpha(i)}{12}+\frac{\alpha(i)}{12}+\frac{\alpha(i)}{12}\\
&=\frac{\alpha(i)}{4},
\end{align*}
where in the inequality above, we have applied
\begin{align*}
\alpha(i)&=\frac{1-\beta(i)}{24\sqrt{2}L\beta(i)}\leq \frac{1-\beta(i)}{24L\frac{1}{2}}\leq\frac{1-\beta(i)}{12L},\\
\alpha(i)&=\frac{1-\beta(i)}{24\sqrt{2}L\beta(i)}\leq \frac{1-\beta(i)}{12L\beta^2(i)},\\
\alpha(i)&=\frac{1-\beta(i)}{24\sqrt{2}L\beta(i)}\leq \frac{1-\beta_1}{24L\beta_1}\leq \frac{1-\beta_1}{12L}.
\end{align*}

Therefore, \eqref{equ: 1 new} is true and 
\begin{align}
\label{equ: R 1 bound new}
R_{1,i}\geq \frac{\alpha(i)}{2}.
\end{align}

Now let us turn to $R_{2,i}$. By \eqref{equ: R_2 new} and \eqref{equ: c 1 bound new} we know that
\begin{align*}
R_{2,i} &= \alpha(k){\rho_{0,i}}L^2A_1^2\EE[\|m^{i-1} - \sum_{j=1}^{i-1} b_{i-1,j} g^j\|^2]+\frac{1}{2}L\alpha^2(i)\sigma^2+2c_1\alpha^2(i)\EE[\|m^{i} - \sum_{j=1}^{i} b_{i,j} g^j\|^2].\nonumber\\
&\leq \alpha(k){\rho_{0,i}}L^2A_1^2\EE[\|m^{i-1} - \sum_{j=1}^{i-1} b_{i-1,j} g^j\|^2]+\frac{1}{2}L\alpha^2(i)\sigma^2+\frac{L}{2(1-\beta_1)}\alpha^2(i)\EE[\|m^{i} - \sum_{j=1}^{i} b_{i,j} g^j\|^2].\nonumber
\end{align*}
Since $\rho_{0,i} = \frac{1-\beta(i)}{2L\alpha(i)}$ and $\frac{\alpha(i)\beta(i)}{1-\beta(i)}\equiv A_1$, we have
\begin{align*}
\begin{split}
R_{2,i}&\leq \frac{1}{2}L\alpha^2(i)\beta^2(i)\frac{1}{1-\beta(i)}\EE[\|m^{i-1} - \sum_{j=1}^{i-1} b_{i-1,i} g^j\|^2]+\frac{1}{2}L\alpha^2(i)\sigma^2\\
&\quad +\frac{L}{2(1-\beta_1)}\alpha^2(i)\EE[\|m^{i} - \sum_{j=1}^{i} b_{i,j} g^j\|^2].
\end{split}
\end{align*}
By applying Lemmas \ref{lem: m^k for multistage} and \ref{lem: m^k-1 for multistage}, we further have
\begin{align}
\label{equ: R 2 bound new}
\begin{split}
R_{2,i}&\leq 12 L\alpha^2(i)\beta^2(i)\frac{\beta_1}{\sqrt{\beta_n+\beta^2_n}}\sigma^2+ \frac{3}{2}L\alpha^2(i)\sigma^2.
\end{split}
\end{align}
By putting \eqref{equ: R 1 bound new} and \eqref{equ: R 2 bound new} into \eqref{equ: 2 new} with $k=T_1+T_2+ \dots +T_n$, we obtain
\begin{align*}
\sum_{l=1}^n\frac{\alpha_l}{2}\sum_{i=T_1+ \dots +T_{l-1}+1}^{T_1+ \dots +T_l} \E[\|g^i\|^2]
&\leq L^1+ \sum_{l=1}^n T_l\left( 12L\alpha^2_l\beta^2_l\frac{\beta_1}{\sqrt{\beta_n+\beta^2_n}}\sigma^2+\frac{3}{2}L\alpha^2_l\sigma^2\right).
\end{align*}
Dividing both sides by $\frac{1}{2}n A_2\equiv \frac{1}{2}n\alpha_l T_l$ gives
\begin{align*}
\begin{split}
&\frac{1}{n}\sum_{l=1}^n\frac{1}{T_l}\sum_{i=T_1+ \dots +T_{l-1}+1}^{T_1+ \dots +T_l} \E[\|g^i\|^2]\\
&\leq \frac{2\left(f(x^1)-f^*\right)}{n A_2} +\frac{1}{n}\sum_{l=1}^n \left(24 \beta_{l}^2\frac{\beta_1}{\sqrt{\beta_{n}+\beta_{n}^2}} L \alpha_l\sigma^2+3L\alpha_l \sigma^2\right) \\
&= \mathcal{O}\left(\frac{f(x^1)-f^*}{n A_2}\right)+\mathcal{O}(\frac{1}{n}\sum_{l=1}^n L\alpha_l\sigma^2).
\end{split}  
\end{align*}



%

%


\section{Details of computational infrastructure}
\label{app: details and other experiments}
All experiments were performed on a computing server with Intel(R) Core(TM) i9-9940X CPU @ 3.30GHz and NVidia GeForce RTX 2080 P8. The weights of the neural networks are initialized by the default, random initialization routines.

\end{document}